\documentclass[11pt]{article}
\usepackage{graphicx,setspace}
\usepackage{amsmath,amssymb,amsthm}
\usepackage{psfig,psfrag}
\newcommand{\be}{\begin{equation}}
\newcommand{\ee}{\end{equation}}
\newcommand{\bea}{\begin{eqnarray}}
\newcommand{\beas}{\begin{eqnarray*}}
\newcommand{\no}{\nonumber}
\newcommand{\eea}{\end{eqnarray}}
\newcommand{\eeas}{\end{eqnarray*}}

\newcommand{\pr}[1]{\mathsf{P}\left( #1 \right)}
\newcommand{\selfnote}[1]{\textbf{
    \\ ***Note to Self:} {\textsf{#1}} \textbf{***} \\}
\newcommand{\EXP}[1]{\mathsf{E}\!\left(#1\right) }

\newcommand{\remove}[1]{}

\newtheorem{thm}{Theorem}[section]
\newtheorem{cor}[thm]{Corollary}
\newtheorem{lem}[thm]{Lemma}
\newtheorem{prop}[thm]{Proposition}
\newtheorem{remark}[thm]{\bf Remark}
\newtheorem{defn}[thm]{\bf Definition}

\def\sec{\setcounter{equation}{0}}

\newcounter{cnt1}

\newcounter{cnt3}
\newcommand{\blr}{\begin{list}{$($\roman{cnt1}$)$}
 {\usecounter{cnt1} \setlength{\topsep}{0pt}
 \setlength{\itemsep}{0pt}}}
\newcommand{\bla}{\begin{list}{$($\betaph{cnt2}$)$}
 {\usecounter{cnt2} \setlength{\topsep}{0pt}
 \setlength{\itemsep}{0pt}}}
\newcommand{\bln}{\begin{list}{$($\arabic{cnt3}$)$}
 {\usecounter{cnt3} \setlength{\topsep}{0pt}
 \setlength{\itemsep}{0pt}}}
\newcommand{\el}{\end{list}}

\def\rar{\rightarrow}
\def\la{\langle}
\def\ra{\rangle}

\newcommand{\ep}{\epsilon}
\newcommand{\al}{\alpha}

\newcommand{\lam}{\lambda}

\newcommand{\sg}{\sigma}

\def\mR{\mathbb{R}}
\def\mZ{\mathbb{Z}}
\def\mN{\mathbb{N}}
\def\mE{\mathbb{E}}
\def\mL{\mathbb{L}}

\def\cC{{\mathcal C}}
\def\uG{{\underline G}}

\def\Pdl{{\mathcal P}_{n}^{(1)}}
\def\Pdlx{{\mathcal P}_{n}^{(1,x)}}
\def\Pdly{{\mathcal P}_{n}^{(1,y)}}
\def\Pdm{{\mathcal P}_{n}^{(2)}}
\def\Pdc{{\mathcal P}_{cn}^{(2)}}
\def\Pn{{\mathcal P}_{n}}

\def\P{\mathcal P}
\def\tG{\tilde{G}}
\def\1{\mathbf{1}}
\def\half{\frac{1}{2}}
\begin{document}

%%%%%%%%%%%%%%%%%%%%%%%%%%%%%%%%TITLE%%%%%%%%%%%%%%%%%%%%%%%%%%%%%%%%%%%%%%%%
\begin{titlepage}
\begin{center}
{\bf Percolation and Connectivity in AB Random Geometric Graphs} \\
\vspace{0.2in} {Srikanth K. Iyer \footnote{corresponding author: skiyer@math.iisc.ernet.in}$^,$\footnote{Research Supported in part by UGC SAP -IV and DRDO grant No. DRDO/PAM/SKI/593}}\\
Department of Mathematics,
Indian Institute of Science, Bangalore, India. \\
\vspace{0.1in}
D. Yogeshwaran \footnote{Supported in part by a grant from EADS, France.}\\
INRIA/ENS TREC, Ecole Normale Superieure, Paris, France.
\end{center}
\vspace{0.1in}
%
%\begin{center} {\bf Preprint} \end{center}
\sloppy
\begin{center} {\bf Abstract} \end{center}

%\begin{center} \parbox{4.8in}

{Given two independent Poisson point processes $\Phi^{(1)},\Phi^{(2)}$  in $\mR^d$, the continuum $AB$ percolation 
model is the graph with points of $\Phi^{(1)}$ as vertices and with edges between any pair of points for which the intersection of balls of radius $2r$ centred at these points contains at least one point of $\Phi^{(2)}$. This is a generalization of the $AB$ percolation model on discrete lattices. We show the existence of percolation for all $d \geq 2$ and derive bounds for a critical intensity. We also provide a characterization for this critical intensity when $d = 2$. To study the connectivity problem, we consider independent Poisson point processes of intensities $n$ and $cn$ in the unit cube. The $AB$ random geometric graph is defined as above but with balls of radius $r$. 
We derive a weak law result for the largest nearest neighbour distance and almost sure asymptotic bounds for the connectivity threshold. 
} \\
%
%\vspace{0.1in} 
\begin{center}
\today
\end{center}

%\rule[1mm]{4in}{.4mm} \\
\vspace{0.1in}
{\sl AMS 1991 subject classifications}: \\
\hspace*{0.5in} Primary:   60D05, 60G70;
\hspace*{0.5in} Secondary:  05C05, 90C27 \\
{\sl Keywords:} Random geometric graph, percolation, connectivity, wireless networks, secure communication.

\end{titlepage}

\section{Introduction}
\label{sec:intro}
\sec
The Bernoulli (site) percolation model on a graph $G := (V,E)$ is defined as follows : Each vertex $v \in V$ of the graph is retained with a probability $p$ or removed, with probability $1-p$, along with all the edges incident to that vertex, independently of other vertices. The model is said to percolate if the random sub-graph resulting from the deletion procedure contains an infinite connected component. The classical percolation model is the Bernoulli bond percolation model with the difference being that the deletion procedure is applied to the edges instead of the vertices. \cite{Grimmett99} is an excellent source for the rich theory on this classical percolation model. A variant of the Bernoulli site percolation model that has been of interest is the $AB$ percolation model. This model was first studied in \cite{Halley80,Halley83,Sevsek83}. The model is as follows : Given a graph $G$, each vertex is marked independently of other vertices either $A$ or $B$. Edges between vertices with similar marks ($A$ or $B$) are removed. The resulting random sub-graph is the $AB$ graph model. If the $AB$ graph contains an infinite connected component with positive probability, we say that the model percolates. An infinite connected component in the $AB$ graph is equivalent to an infinite path of vertices in $G$ with marks alternating between $A$ and $B$. This model has been studied on lattices and some related graphs. The $AB$ percolation model behaves quite differently as compared to the Bernoulli percolation model. For example, it is known that $AB$ percolation does not occur in $\mZ^2$ (\cite{Appel87}), but occurs on the planar triangular lattice (\cite{Wierman87}), some periodic two-dimensional graphs (\cite{Scheinerman87}) and the half close-packed graph of $\mZ^2$ (\cite{Wu04}). It is also known that the AB bond percolation does not occur in $\mZ^2$ for $p = \half$ (\cite{Wu04}). See \cite{Wu04,Grimmett99} for further references.

The following generalization of the discrete $AB$ percolation model has been studied on various graphs by Kesten {\it et. al.} (see \cite{Benjamini95,Kesten98,Kesten01}). Mark each vertex or site of a graph $G$ independently as either $0$ or $1$ with probability $p$ and $1-p$ respectively. Given any infinite sequence (referred to as a word) $w \in \{0,1\}^{\infty}$, the question is whether $w$ occurs in the graph $G$ or not. The sentences $(1,0,1,0...),(0,1,0,1..)$ correspond to $AB$ percolation and the sequence $(1,1,1...)$ corresponds to Bernoulli percolation. More generally Kesten {\it et. al.} answer the question whether all (or almost all) infinite sequences (words) occur or not. The graphs for which the answer is known in affirmative are 
$\mZ^d$ for $d$ large, triangular lattice and $\mZ_{cp}^2$, the close-packed graph of $\mZ^2$. Our results provide partial answers to these questions in the continuum.

Our aim is to study a generalization of the discrete $AB$ percolation model to the continuum. We study the problem of percolation and connectivity in such models. For the percolation problem the vertex set of the graph will be a homogenous Poisson point process in $\mR^d$. For the connectivity problem we  will consider a sequence of graphs whose vertex sets will be homogenous Poisson point processes of intensity $n$ in $[0,1]^d$. We consider different models while studying percolation and connectivity so as to be consistent with the literature. This allows for easy comparison with, as well as the use of, existing results from the literature. We will refer to our graphs, in the percolation context as the continuum $AB$ percolation model, and as the $AB$ random geometric graph while investigating the connectivity problem. The continuum percolation model and random geometric graphs where the nodes are of the same 
type are the topics of monographs \cite{Meester96} and \cite{Penrose03} respectively.

Our motivation for the study of $AB$ random geometric graphs comes from applications to wireless communication. In models of ad-hoc wireless networks, the nodes are assumed to be communicating entities that are distributed randomly in space. Edges between any two nodes in the graph represents the ability of the two nodes to communicate effectively with each other. In one of the widely used models, a pair of nodes share an edge if the distance between the nodes is less than a certain cutoff radius $r > 0$ that is determined by the transmission power. Percolation and connectivity thresholds for such a model have been used to derive, for example, the capacity of wireless networks (\cite{Franc07,Gupta00}). Consider a transmission scheme called the frequency division half duplex, where each node transmits at a frequency $f_1$ and receives at frequency $f_2$ or vice-versa (\cite{Tse05}). Thus nodes with transmission-reception frequency pair $(f_1,f_2)$ can communicate only with nodes that have transmission-reception frequency pair $(f_2,f_1)$ that are located within the cutoff distance $r.$ Another example where such a model would be applicable is in communication between communicating units deployed at two different levels, for example surface (or underwater) and in air. Units in a level can communicate only with those at the other level that are within a certain range. A third example is in secure communication in wireless sensor networks with two types of nodes, tagged and normal. Upon deployment, each tagged node broadcasts a key over a predetermined secure channel, which is received by all normal
nodes that are within transmission range. Two normal nodes can then communicate provided there is a tagged node from which both these normal nodes have received a key, that is, the tagged node is within transmission range of both 
the normal nodes.

The rest of the paper is organized as follows. Sections \ref{sec:perc_model} and \ref{sec:rgg_model} provide definitions and statements of our main theorems on percolation and connectivity respectively. Sections \ref{sec:perc_proofs} and \ref{sec:rgg_proofs} contain the proofs of these results.

\section{Percolation in the Continuum $AB$ Percolation Model}
\label{sec:perc_model}
\sec

\subsection{Model Definition}
\label{sec:model_defn_bool}

Let $\Phi^{(1)} = \{X_i\}_{i \geq 1}$ and $\Phi^{(2)} = \{Y_i\}_{i \geq 1}$ be independent Poisson point processes 
in $\mR^d,$ $d \geq 2$, with intensities $\lam$ and $\mu$ respectively. Let
the Lebesuge measure and the Euclidean metric on $\mR^d$ be denoted by $\| \cdot \|$  and $| \cdot |$ respectively. Let $B_x(r)$ denote the closed ball of radius $r$ centred at $x \in \mR^d$.

By percolation in a graph, we mean the existence of an infinite connected component in the graph. The standard continuum percolation model (introduced in \cite{Gilbert61}), also called the continuum percolation model or Gilbert disk graph, is defined as follows.
\begin{defn} 
\label{def:basic_bool}
Define the continuum percolation model, $\tG(\lam,r) := (\Phi^{(1)}, \tilde{E}(\lam,r))$ to be the graph with
vertex set $\Phi^{(1)}$ and edge set
\[ \tilde{E}(\lam,r) = \{ \la X_i,X_j \ra : X_i,X_j \in \Phi^{(1)}, |X_i-X_j| \leq 2r\}. \]
For fixed $r>0$, define the critical intensity of the continuum percolation model as follows :
\begin{equation} 
\lam_c(r) := \sup \left\{ \lam > 0: \pr{\tG(\lam,r) \mbox{ percolates } } = 0 \right\}.
\label{eqn:crit_lam}
\end{equation} 
\end{defn}
The edges in all the graphs that we consider are undirected, that is, $\la X_i, X_j \ra \equiv \la X_j, X_i \ra.$ We will use the notation $X_i \sim X_j$ to denote existence of an edge between $X_i$ and $X_j$ when the underlying 
graph is unambiguous. For the continuum percolation model defined above (\cite{Meester96}) it is known that $0 < \lam_c(r) < \infty$. Topologically, percolation in the above model is equivalent to existence of an unbounded connected subset in $\cup_{X \in \Phi^{(1)}}B_X(r).$ Also, by zero-one law, one can deduce that the probability of percolation is either zero or one. 

A natural analogue of this model to the $AB$ set-up would be to consider a graph with vertex set $\Phi^{(1)}$ where each vertex is independently marked either $A$ or $B$. We will consider a more general model from which results for the above model will follow as a corollary.
\begin{defn}
\label{defn:AB_Bool}
The continuum $AB$ percolation model $G(\lam,\mu,r) := (\Phi^{(1)},E(\lam,\mu,r))$ 
is the graph with vertex set $\Phi^{(1)}$ and edge set
$$E(\lam,\mu,r) := \{ \la X_i, X_j \ra : X_i,X_j \in \Phi^{(1)}, |X_i-Y| \leq 2r, |X_j-Y| \leq 2r, \mbox{ for some } 
Y \in \Phi^{(2)} \}. $$
Let $\theta(\lam,\mu,r) = \pr{G(\lam, \mu, r) \; \mbox{percolates}}.$  For a fixed $\lam, r > 0$, define the 
critical intensity $\mu_c(\lam, r)$ by
\begin{equation}
\label{eqn:critical_mu} 
\mu_c(\lam,r) := \sup\{\mu > 0: \theta(\lam,\mu,r) = 0 \}.
\end{equation}
\end{defn}
%. 
It follows from zero-one law that $\theta(\lam,\mu,r) \in \{0,1\}$. We are interested in characterizing the 
region formed by $(\lam,\mu,r)$ for which  $\theta(\lam,\mu,r) = 1$. 
\subsection{Main Results}
\label{sec:main_results_bool}

We start with some simple lower bounds for the critical intensity $\mu_c(\lam,r)$. 
\remove{We first make two simple observations.
For any $\lam$, coupling $C(\lam,\mu,r)$ as a subgraph of $C(\lam+\mu,r)$, it is easy to see that $\mu_c(\lam,r) + \lam \geq \lam_c(r)$. Thus for  $\lam < \lam_c(r)$, we have the lower bound $\mu_c(\lam,r) \geq \lam_c(r) -  \lam$.
On the other hand $<X_i,X_j> \; \in E(\lam,\mu,r)$ implies that $|X_i-X_j| \leq 4r$. Hence, $C(\lam,\mu,r)$ has an infinite component only if $\cC(\lam,2r)$ has an infinite component. Thus $\mu_c(\lam,r) = \infty$ if $\lam \leq \lam_c(2r)$.} 
\begin{prop} Fix $\lam,r > 0$. Let $\lam_c(r)$, $\mu_c(\lam,r)$ be the critical
intensities as in (\ref{eqn:crit_lam}) and 
(\ref{eqn:critical_mu}), respectively. Then
\begin{enumerate}
\item $\mu_c(\lam,r) \geq \lam_c(r) -  \lam$, $\;\;$ if $\;\; \lam_c(2r) < \lam < \lam_c(r)$, $\;\;$ and
\item $\mu_c(\lam,r) = \infty$,  $\;\;$ if  $\;\; \lam < \lam_c(2r)$.
\end{enumerate}
\label{prop:simple_bd}
\end{prop}
The second part of the above proposition holds true for $\lam = \lam_c(2r)$ provided that $\tG(\lam_c(2r),2r)$ does not percolate. This has been proven for $d = 2$ (\cite[Theorem 4.5]{Meester96}) and for all but at most finitely many $d$ (\cite{Tanemura96}). The next question is whether $\mu_c(\lam,r) < \infty$ if $\lam > \lam_c(2r)$. We answer this in affirmative for $d = 2$. 
\begin{thm}
\label{thm:perc_threshold}
Let $d=2$ and $r > 0$ be fixed. Then for any $\lam > \lam_c(2r)$, we have $\mu_c(\lam,r) < \infty$.
\end{thm}
Thus the continuum $AB$ percolation model exhibits a {\it phase transition} in the plane. However, the above theorem does not tell us how to choose a $\mu$ for a given $\lam > \lam_c(2r)$ for $d=2$ such that $AB$ percolation happens, or if indeed there is a phase transition for $d \geq 3$. We obtain an upper bound for $\mu_c(\lam,r)$ as a special case of a more general result which is the continuum analog of word percolation on discrete lattices described in Section~\ref{sec:intro}. In order to state this result, we need some notation.
\begin{defn}
\label{def:p_c_a}
For each $d \geq 2$, define the critical probabilities $p_c(d)$, and the functions $a(d,r)$ as follows.
\begin{enumerate}
\item For $d =2$, consider the triangular lattice $\mathbb{T}$ (see Figure~\ref{fig:flower}) with edge length $r/2$. Let $p_c(2)$ be the critical probability for the Bernoulli site percolation on this lattice.  Around each vertex 
place a ``flower'' formed by the arcs (see Figure~\ref{fig:flower}) of the $6$ circles, each of radius 
$\frac{r}{2}$ and centred at the mid-points of the $6$ edges adjacent to the vertex. Let $a(2,r)$ be the area of 
a flower. 
\item For $d \geq 3$, let $p_c(d)$ be the critical probability for the Bernoulli site percolation on $\mZ^{*d} := (\mZ^d, \mE^{*d} := \{ <z,z_1> : |z-z_1|_{\infty} = 1\})$, where $|.|_{\infty}$ stands
for the $l_{\infty}-$norm. Define $a(d,r) = (r/2\sqrt{d})^{d}.$ 
\end{enumerate}
\end{defn}
\begin{figure}[h!t]
  %\vspace{-1.3in}
  \centerline{\includegraphics[scale=0.4]{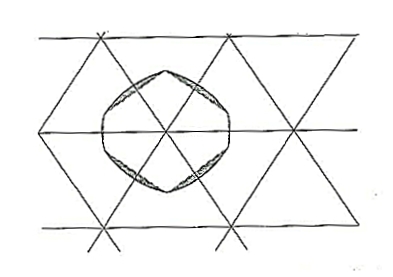}}
  %\vspace{-2.6in}
  \caption{The triangular lattice and flower in $\mR^2$ with area $a(2,r)$. The figure is reproduced from \cite[Fig 3.2]{Meester96}.}
  \label{fig:flower}
\end{figure}
It is known that $p_c(2) = \frac{1}{2}$, and $p_c(d) < 1,$ for $d \geq 3$ (see \cite{Grimmett99}). 
\begin{defn}
\label{defn:word}
For $i=1,\ldots,k,$ let $\Phi^{(i)}$ be independent Poisson point processes of intensities $\lam_i > 0$. Fix $(r_1,\ldots,r_k) \in \mR^k_+.$ A word $\omega := \{w_i\}_{i \geq 1} \in \{1,2,\ldots,k\}^{\mN}$ is said to occur if there exists a sequence of distinct elements $\{X_i\}_{i \geq 1} \subset \mR^d$, such that $X_i \in \Phi^{(w_i)}$, and $|X_i-X_{i+1}| \leq r_{w_i}+r_{w_{i+1}}$, for $i \geq 1$. 
\end{defn}
\begin{prop}
\label{prop:upper_bd}
For any $d \geq 2$, let $p_c(d)$, $a(d,r)$ be as in Definition~\ref{def:p_c_a}. Fix $k \in \mN$ and let $(r_1,\ldots,r_k) \in \mR^k_+$. Also for $i=1,\ldots,k,$
let $\Phi^{(i)}$ be independent Poisson point processes of intensities $\lam_i > 0$. Set $r_0 = \inf_{1 \leq i,j \leq k}\{r_i+r_j\}.$ 
If $\prod_{i=1}^k(1-e^{-\lam_ia(d,r_0)}) > p_c(d)$, then almost surely, every word occurs. 
\end{prop}
The following corollary, the proof of which is given in Section \ref{sec:perc_proofs}, gives an upper bound for $\mu_c(\lam,r)$ for large $\lam$.
\begin{cor}
Suppose that $d \geq 2$, $r > 0$, and $\lam > 0$ satisfies 
\[  \lam > - \; \frac{\log \left(1 - p_c(d) \right)}{a(d,2r)}, \]
where $p_c(d)$, $a(d,r)$ are as in Definition~\ref{def:p_c_a}. Let $\mu_c(\lam,r)$ be the critical
intensity as in (\ref{eqn:critical_mu}). Then
\begin{equation}
\label{eqn:mu_c_ub}
\mu_c(\lam,r) \leq - \frac{1}{a(d,2r)} \log \left[ 1 - \left( \frac{p_c(d)}{1 - e^{-\lam a(d,2r)}} \right) \right].
\end{equation}
%\inf \left\{\mu > 0: \left( 1-e^{-\lam a(d,2r)} \right) \left( 1-e^{-\mu a(d,2r)} \right) > p_c(d) \right\} < \infty. %\]

\remove{Let $G(\lam,\mu,r)$ be the $AB$ percolation model as in
Definition~\ref{defn:AB_Bool}, and let $\mu_c(\lam,r)$ be the critical
intensity as in (\ref{eqn:critical_mu}). For any $\lam$ satisfying $a(d,2r)\lam > -\log \left(1 - p_c(d) \right) $, we have that $\mu_c(\lam,r) \leq \inf \{\mu > 0:(1-e^{-\lam a(d,2r))})(1-e^{-\mu a(d,2r)}) > p_c(d)\}.$}
\label{cor:upper_bound}
\end{cor}

\begin{remark}
A simple calculation (see \cite{Meester96}, pg.88) gives $a(2,2) \simeq 0.8227$, and $$-(a(2,2))^{-1} \log (1 - p_c(2)) \simeq 0.843.$$ Using these we obtain from Corollary~\ref{cor:upper_bound} that $\mu_c(0.85,1) < 6.2001$.
\end{remark}
\begin{remark}
It can be shown that the number of infinite components in the continuum $AB$ percolation model is at most one, 
almost surely. The proof of this fact follows along the same lines as the proof in the continuum percolation model (see \cite[Proposition 3.3, Proposition 3.6]{Meester96}), since it relies on the ergodic theorem and the topology of infinite components but not on the specific nature of the infinite components.
\end{remark}
The proposition above can be used to show existence of $AB$ percolation in the natural analogue of the discrete 
$AB$ percolation model (refer to the two sentences above Definition~\ref{defn:AB_Bool}). Recall that $\Phi^{(1)}$ 
is a Poisson point process in $\mR^d$ of intensity $\lam > 0$. Let $\{m_i\}_{i \geq 1}$ be a sequence of i.i.d. 
marks distributed as $m \in \{A,B\}$, with $P(m=A) = p = 1 - P(m=B)$. Define the point processes $\Phi^A, \Phi^B$ as 
\[ \Phi^A := \{X_i \in \Phi^{(1)} : m_i = A \}, \qquad \Phi^B := \Phi^{(1)} \setminus \Phi^A. \] 
\begin{defn}
\label{defn:AB_Bool1}
For any $\lam, r > 0$, and $p \in (0,1)$, let $\Phi^A$ and $\Phi^B$ be as defined above. Let $\widehat{G}(\lam,p,r) := (\Phi^A, \widehat{E}(\lam,p,r))$ be the graph with vertex-set $\Phi^A$ and edge-set 
$$\widehat{E}(\lam,p,r) := \{<X_i,X_j> : X_i, X_j \in \Phi^A, \;    |X_i-Y| \leq 2r, |X_j-Y| \leq 2r, \; \mbox{ for some } Y \in \Phi^B \}.$$
\end{defn} 
\begin{cor}
\label{cor:AB_perc}
Let $\widehat{\theta}(\lam,p,r) := P(\widehat{G}(\lam,p,r) \; \mbox{percolates})$. Then for any $\lam$ satisfying 
\[ \lam > - \; \frac{2 \log \left(1-\sqrt{p_c(d)} \right)}{a(d,2r)},\] 
there exists a $p(\lam) < \frac{1}{2}$, such that $\widehat{\theta}(\lam,p,r) = 1$, for all $p \in (p(\lam),1-p(\lam))$.
\end{cor}
\remove{
The final result in this section asserts that when percolation happens in the $AB$ Boolean model, the giant 
component will be unique.
\begin{prop}
\label{prop:uniqueness}
For any non-negative values of the parameters $\lam, \mu , r$, there is at most one infinite component in $C(\lam,\mu,r)$ almost surely.
\end{prop}
}
\section{Connectivity in $AB$ Random Geometric Graphs}
\label{sec:rgg_model}
\sec

\subsection{Model Definition}
\label{sec:model_defn_rgg}

The set up for the study of connectivity in $AB$ random geometric graphs is as follows.  For each $n \geq 1,$ let $\Pdl$ and $\Pdm$ be independent homogenous Poisson point processes in $U = [0,1]^d$, $d \geq 2$, of intensity $n$. We also nullify some of the technical complications arising out of boundary effects by choosing to work with the toroidal metric on the unit cube, defined as
\begin{equation}
d(x,y) := \inf\{|x-y+z|: z \in \mathbb{Z}^d \}, \qquad x,y \in U.
\label{eqn:tm}
\end{equation}
%
%$B(x,r)$ will now denote the ball of radius $r$ around the point $x \in U$ with %respect to the metric $d.$
%
\begin{defn}
\label{defn:AB_RGG}
For any $m, n \geq 1$, the $AB$ random geometric graph 
$G_n(m,r)$ is the graph with vertex set $\Pdl$ and edge set
\[ E_n(m,r):= \{ \la X_i,X_j \ra : X_i,X_j \in \Pdl, d(X_i, Y) \leq r, 
d(X_j,Y) \leq r, \mbox{ for some } Y \in \mathcal{P}^{(2)}_m \}. \]
\end{defn}
Our goal in this section is to study the {\it connectivity threshold} in the sequence of graphs $G_n(cn,r)$ 
as $n \to \infty$ for $c > 0$. The constant $c$ can be thought of as a measure of the relative denseness or 
sparseness of $\Pdl$ with respect to $\Pdc$ (see Remark~\ref{rem:rem_c} below). We will also prove
a distributional convergence result for the critical radius required to eleminate isolated nodes. To this
end we introduce the following definition.
\begin{defn}
\label{defn:iso_lnnr}
For each $n \geq 1$, let $W_n(r)$ be the number of isolated nodes, that is,
vertices with degree zero in $G_n(cn,r)$, and define 
the {\it largest nearest neighbor radius} as
$$M_n := \sup \{r \geq 0: W_n(r) > 0\}.$$ 
\end{defn}
\subsection{Main Results}
\label{sec:main_results_rgg}
Let $\theta_d := \|B_O(1)\|$ be the volume of the $d$-dimensional unit closed ball centered at the origin. 
For any $\beta > 0$, and $n \geq 1$, define the sequence of cut-off functions,
\begin{equation}
\label{ass:lt_regime}
r_n(c,\beta) = \left( \frac{\log(n/\beta)}{cn\theta_d} \right)^{\frac{1}{d}},
\end{equation}
and let
\begin{equation}
r_n(c) = r_n(c,1).
\label{eqn:r_n(c,1)}
\end{equation}
\remove{It follows from (\ref{eqn:exp_isoln}) that with the choice of $r_n$ as in (\ref{ass:lt_regime}),
\begin{equation}
\EXP{W_n(r_n)} \to \beta, \qquad \mbox{ as } n \to \infty.
\label{eqn:cgs_exp_iso_nodes}
\end{equation}
}
Let $e_1 := (1,0,\ldots,0) \in \mR^d$ be the unit vector in the first coordinate direction. For $d \geq 2$ 
and $u,s > 0$, define 
\begin{equation}
\label{eqn:defn_eta}
\eta(u,s) := \frac{\|B_O(u^{\frac{1}{d}}) \cap B_{s^{\frac{1}{d}}e_1}(u^{\frac{1}{d}}) \|}{\theta_d u}.
\end{equation}
For $s \leq 2u$, we have (see \cite[(7.5)]{Goldstein10} and \cite[(6)]{Moran73})
\begin{equation}
\label{eqn:exp_eta}
\eta(u,s) = 1 - \frac{\theta_{d-1}}{\theta_d} \int_0^{(\frac{s}{u})^{\frac{1}{d}}} \left( 1 -\frac{t^2}{4} \right)^{\frac{d-1}{d}} dt.     
\end{equation}
If $s \geq 2u$, then $\eta(u,s) = 0$. Since, the intersection $B_O(u^{\frac{1}{d}}) \cap B_{s^{\frac{1}{d}}e_1}(u^{\frac{1}{d}})$ always contains a ball of diameter $(2u^{\frac{1}{d}} - s^{\frac{1}{d}})$, 
we get the following lower bound : 
\begin{equation}
\label{eqn:bds_eta}
\eta(u,s) \geq \left(1 - \half  \left( \frac{s}{u} \right)^{\frac{1}{d}} \right)^d. 
\end{equation}
The next theorem gives asymptotic bounds for a strong connectivity threshold in $AB$ random geometric graphs. Asymptotics for the strong connectivity threshold was one of the more difficult problems in the theory of random geometric graphs. We will take $\beta = 1$ in (\ref{ass:lt_regime}) and work with the cut-off functions $r_n(c)$ as defined in 
(\ref{eqn:r_n(c,1)}). 
\remove{ 
Define the function $\eta : \mR_+^2 \to \mR$ by 
\begin{equation}
\eta(a,c) =
\begin{cases}
 \frac{1}{\pi} \left[ 2\phi\left( \half \left( \frac{c}{a} \right)^{\frac{1}{2}}
\right) - \sin \left(2\phi\left( \half \left( \frac{c}{a} \right)^{\frac{1}{2}}
\right) \right) \right] & \mbox{if } d = 2 \\
\left(1 - \half  \left( \frac{c}{a} \right)^{\frac{1}{d}}
 \right)^d & \mbox{if } d \geq 3,
\end{cases}
\label{eqn:defn_eta}
\end{equation}
where $\phi(a) = \arccos (a)$.}
Define the function $\al: \mR_+ \to \mR$ by
\begin{equation}
\al (c) := \inf \{a : a\eta(a,c) > 1 \}.
\label{eqn:defn_alpha}
\end{equation} 
From (\ref{eqn:exp_eta}), it is clear that for fixed $c > 0$, $\eta(a,c)$ is increasing in $a$ for $a > c/2$
and converges to 1 as $a \to \infty$ and hence $a(c) < \infty$. 
From the bound (\ref{eqn:bds_eta}), we get that 
$$\left(1 + \frac{c^{\frac{1}{d}}}{2} \right)^d \eta\left(\left(1 + \frac{c^{\frac{1}{d}}}{2} \right)^d,c \right) 
\geq 1,$$ 
for $d \geq 2$. Thus we have the bound $\al(c) \leq \left(1 + \frac{c^{\frac{1}{d}}}{2} \right)^d$ for $d \geq 2$. 
%and $a^*_n(c) := \inf \{a :  GR(n,cn,a^{\frac{1}{d}}r_n) \, \mbox{is connected} \}.$
%
\begin{thm}
\label{thm:conn}
Let $\al(c)$ be as defined in (\ref{eqn:defn_alpha}) and $r_n(c)$ be as defined in (\ref{eqn:r_n(c,1)}). 
Define $\al^*_n(c) := \inf \{a :  G_n(cn,a^{\frac{1}{d}}r_n(c)) \, \mbox{is connected} \}.$ Then for any $c > 0$, almost surely, 
\begin{equation} 
\label{eqn:connect}
1 \leq \liminf_{n \to \infty} \al^*_n(c) \leq \limsup_{n \to \infty} \al^*_n(c) \leq \al(c).
%\label{eqn:lb_connect}
%\end{equation}
%
%
%\begin{equation}
%\label{eqn:ub_connect}
\end{equation}
\end{thm}
As is obvious, the bounds are tight for $c$ small enough. We derive the lower bound by covering the space with disjoint circles and showing that at least one of them contains an isolated node. For the upper bound, we couple the AB random geometric graph with a random geometric graph and use the connectivity threshold for the random geometric graph (see Theorem \ref{thm:conn_rgg}).
\remove{ 
Let $\phi(x) = \arccos(x)$. For $d = 2$, define 
\begin{equation} 
A(c) = \pi^{-1} \left[ 2\phi\left( \frac{c^{\frac{1}{2}}}{2} \right) - 
\sin \left( 2 \phi\left( \frac{c^{\frac{1}{2}}}{2} \right) \right) \right].
\label{eqn:defn_A(c)}
\end{equation}
%
%:= \frac{\|B_O(r_n(c,\beta)) \cap B_{r_n(1,\beta)e_1}(r_n(c,\beta)) \|}{\pi r_n(c,\beta)^2}$ where $e_1 = (1,0)$ is the unit vector in $x$-direction. 
%
}
%\selfnote{Yogesh : Can you add some lines here as the transition seems a bit abrupt ? }

In order to derive the asymptotic distribution of the critical radius required to eleminate isolated nodes, we
need to first find conditions on the parameters $c$ and $\beta$ in (\ref{ass:lt_regime}) so that the expected 
number of isolated nodes will stabilize in the limit. This is the content of Lemma~\ref{lem:cgs_exp_isoln}.

Set $\eta(s) := \eta(1,s)$ and note that $\eta(u,s) = \eta(\frac{s}{u})$ by $(\ref{eqn:exp_eta}).$ Define the 
constant $c_0$ as follows :
\begin{equation} c_0 := \left\{ \begin{array}{ll} 
\sup \{c : \eta(c) + \frac{1}{c} > 1 \} &  \mbox{if } d = 2 \\
1 & \mbox{if } d \geq 3. 
\end{array} \right.
\label{eqn:defn_c_0}
\end{equation}
\remove{
Since $\phi(\frac{c^{\frac{1}{2}}}{2})$ is a decreasing function in $c$ and $cos(x) \leq 1 \, \, \forall x \in [0,2\pi]$, we have that the derivative of $A(c)$, $A'(c) = 2\frac{\phi'(c^{\frac{1}{2}}/2)}{c}\left(1 - \cos(2\phi(\frac{c^{\frac{1}{2}}}{2}))\right) \leq 0$ for $c \in [0,4].$}
From (\ref{eqn:exp_eta}), it is clear that $\eta(c) + \frac{1}{c}$ is decreasing in $c$. Hence $1 < c_0 < 4$ 
for $d = 2$ as $\eta(1) > 0$ and $\eta(4) = 0$. The first part of Lemma~\ref{lem:cgs_exp_isoln} shows that for 
$c < c_0$, the above choice of radius stabilizes the expected number of isolated nodes in $G_n(cn,r_n(c,\beta))$ 
as $n \to \infty$. The second part shows that the assumption $c < c_0$ is not merely technical. The lemma also 
suggests a {\it phase transition} at some $\tilde{c} \in  [1,2^d]$, in the sense that, the expected number of 
isolated nodes in $G_n(cn,r_n(c,\beta))$ converges to a finite limit for $c < \tilde{c}$ and diverges for 
$c > \tilde{c}$.

\begin{lem}
\label{lem:cgs_exp_isoln}
For any $\beta, c > 0$, let $r_n(c,\beta)$ be as defined in (\ref{ass:lt_regime}), and $W_n(r_n(c,\beta))$ be the number of isolated nodes in $G_n(cn, r_n(c,\beta))$. Let $c_0$ be as defined in (\ref{eqn:defn_c_0}).
Then as $n \rar \infty$, 
\begin{enumerate}
%\item $\EXP{W_n(r_n(c,\beta))} \to \beta$ for $ c < 1$.
\item $\EXP{W_n(r_n(c,\beta))} \to \beta$ for $ c < c_0$, and 
\item $\EXP{W_n(r_n(c,\beta))} \to \infty$ for $c > 2^d$.  
\end{enumerate}
\end{lem}
For $c < c_0$, having found the radius that stabilizes the mean number of isolated nodes, the next theorem shows that  the number of isolated nodes and the largest nearest neighbour radius in $G_n(cn,r_n(c,\beta))$ converge in distribution as $n \to \infty$. Let $\stackrel{d}{\rar}$ denote convergence in distribution and $Po(\beta)$ denote a Poisson random variable with mean $\beta$.
\begin{thm}
\label{thm:lnnd}
Let $r_n(c,\beta)$ be as defined in (\ref{ass:lt_regime}) with $\beta > 0$ 
and $0 < c < c_0$. Then as $n \rar \infty$,
%
%\begin{eqnarray}
\begin{equation}
\label{eqn:po_convg} W_n(r_n(c,\beta))  \stackrel{d}{\rar} Po(\beta), 
\end{equation}
%
%and
%
\begin{equation}
\label{eqn:lnnd} \pr{M_n \leq r_n(c,\beta)} \rar e^{-\beta}.
\end{equation}
\end{thm}
%
%Let $\phi(a) = \arccos(a)$ 
%\arccos(\frac{\sqrt{c}}{2\sqrt{a}})$ and 

\begin{remark}
\label{rem:rem_c}
For any locally finite point process $\mathcal{X}$ (for example $\Pdl$ or $\Pdm$), we denote the number of points of $\mathcal{X}$ in $A, \; A \subset \mR^d$ by $\mathcal{X}(A)$. Define 
$$W^0_n(c,r) = \sum_{Y_i \in \Pdc } \1[\Pdl(B_{Y_i}(r))=0],$$ 
that is, $W^0_n(c,r)$ is the number of $\Pdc$ nodes isolated from $\Pdl$ nodes. From Palm calculus for Poisson point processes (Theorem 1.6, \cite{Penrose03})
and the fact that the metric is toroidal, we have
\[ \EXP{W^0_n(c,r_n(c,\beta))} = cn \int_{U} \pr{\Pdl(B_{x}(r))=0 } dx = cn \exp(-n \theta_d r_n(c,\beta)^d). \]
Substituting from (\ref{ass:lt_regime}) we get
\begin{equation}
\lim_{n \rar \infty}\EXP{W^0_n(c,r_n(c,\beta))} =
\begin{cases}
0 & \mbox{ if } c < 1 \\
\beta & \mbox{ if } c =1 \\
\infty & \mbox{ if } c > 1.
\end{cases}
\end{equation}
Thus there is a trade off between the relative density of the nodes and the radius required to stabilise the expected number of isolated nodes.
\end{remark}
\section{Proofs for Section~\ref{sec:perc_model}}
\label{sec:perc_proofs}
\sec

{\bf Proof of Proposition \ref{prop:simple_bd}}

{\sc (1).} Recall from Definition~\ref{defn:AB_Bool} the graph $G(\lam,\mu,r)$ 
with vertex set $\Phi^{(1)}$ and edge set $E(\lam,\mu,r)$. Consider the graph $\tG(\lam + \mu,r)$ (see Definition~\ref{def:basic_bool}), 
where the vertex set is taken to be $\Phi^{(1)}  \cup \Phi^{(2)}$ and let the
edge set of this graph be denoted by $\tilde{E}(\lam+\mu,r)$.

If $<X_i,X_j> \; \in E(\lam,\mu,r)$, then there exists a $Y \in \Phi^{(2)}$ such that $<X_i,Y>, <X_j,Y> \in \tilde{E}(\lam+\mu,r)$. It follows that $G(\lam,\mu,r)$ has an infinite component only if $\tG(\lam+\mu,r)$ has an infinite component.
Consequently, for any $\mu > \mu_c(\lam,r)$ we have $\mu + \lam > \lam_c(r)$, and hence $\mu_c(\lam,r) + \lam \geq \lam_c(r)$. Thus for any $\lam < \lam_c(r)$, 
we obtain the (non-trivial) lower bound $\mu_c(\lam,r) \geq \lam_c(r) -  \lam$. 

{\sc(2).} Again $<X_i,X_j> \; \in E(\lam,\mu,r)$ implies that $|X_i-X_j| \leq 4r$. Hence, $G(\lam,\mu,r)$ has an infinite component only if $\tG(\lam,2r)$ has an infinite component. Thus $\mu_c(\lam,r) = \infty$ if $\lam \leq \lam_c(2r)$. \qed
\remove{ 1. Set $\mathcal{P}_{\lam+\mu} = \Phi^{(1)} \cup \Phi^{(2)}$ and define $C(\lam + \mu)$ to be the graph with vertex set $\mathcal{P}_{\lam+\mu}$ and edge set $E(\lam + \mu,r) =  \{ \la X_i,X_j \ra : X_i, X_j \in \mathcal{P}_{\lam+\mu} , |X_i-X_j| \leq 2r \}$. The proof follows by coupling a suitable modification of $C(\lam,\mu,r)$ with $C(\lam+\mu,r)$. Define $C^{\prime}(\lam,\mu,r)$ to be the graph with vertex set $\Phi^{(1)} \cup \Phi^{(2)}$ and edge set $E^{\prime}(\lam,\mu,r) := \{ \la X_i,Y_j \ra : X_i \in \Phi^{(1)}, Y_j \in \Phi^{(2)} , |X_i-Y_j| \leq 2r \}$. Clearly $E^{\prime}(\lam,\mu,r) \subset E(\lam+\mu,r)$. Observe that there exists a bijection between components in $C(\lam,\mu,r)$ and $C^{\prime}(\lam,\mu,r)$. If $\theta(\lam,\mu,r) = 1$, then $C(\lam,\mu,r)$ has an infinite component and thus $C^{\prime}(\lam,\mu,r)$ also has an infinite component. Hence $C(\lam+\mu,r)$ percolates implying that $\mu_c(\lam,r) + \lam \geq \lam_c(r)$. }

{\bf Proof of Theorem \ref{thm:perc_threshold}}

Fix $\lam > \lam_c(2r)$. The proof adapts the idea used in \cite{Dousse06} of coupling the continuum percolation model to a discrete percolation model. For $l > 0$, let $l\mL^2$ be the graph with vertex set $l\mZ^2$, the expanded two-dimensional integer lattice, and endowed with the usual graph structure,
that is, $x,y \in l\mZ^2$ share an edge if $|x-y| = l$. Denote the edge-set by $l\mE \, ^2$. For any edge $e \in l\mE \, ^2$ denote the mid-point of $e$ by $(x_e,y_e)$. For every horizontal edge $e$, define three rectangles $R_{ei}, i=1,2,3$ as follows : $R_{e1}$ is the rectangle $[x_e-3l/4,x_e-l/4] \times [y_e-l/4,y_e+l/4]$; $R_{e2}$ is the rectangle $[x_e-l/4,x_e+l/4] \times [y_e-l/4,y_e+l/4]$ and $R_{e3}$ is the rectangle $[x_e+l/4,x_e+3l/4] \times [y_e-l/4,y_e-l/4]$. Let $R_e = \cup_iR_{ei}.$ The corresponding rectangles for vertical edges are defined similarly. The reader can refer to Figure~\ref{fig:1}.

\begin{figure}[h!t]
  \centerline{\includegraphics[scale=0.5]{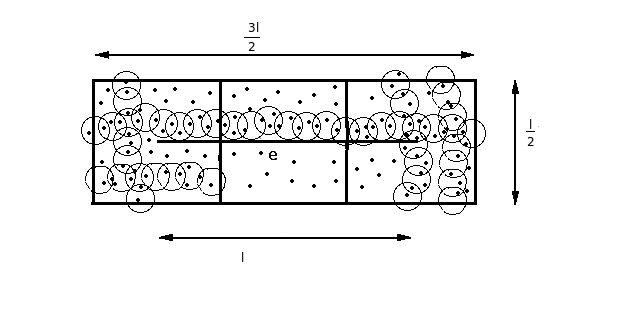}}
  \caption{An horizontal edge $e$ that satisfies the condition for $B_e = 1$. The balls are of radius $2r$, centered at points of $\Phi^{(1)}$ and the adjacent centers are of at most distance $r_1$. The dots are the points of $\Phi^{(2)}$.}
  \label{fig:1}
\end{figure}
Due to continuity of $\lam_c(2r)$ (see \cite[Theorem 3.7]{Meester96}), there exists $r_1 < r$ such that $\lam > \lam_c(2r_1)$. We shall now define some random variables associated with horizontal edges and the corresponding definitions for vertical edges are similar. Let $A_e$ be the indicator random variable for the event that there 
exists a left-right crossing of $R_e$ 
%by a component of $\tG(\lam,2r_1)$ 
and top-down crossings of $R_{e1}$ and $R_{e3}$ by a component of $\tG(\lam,2r_1)$. 
%Suppose that $A_e = 1$. 
Let $C_e$ be the indicator random variable of the event that, $\Phi^{(2)} \cap B_X(2r) \cap B_Y(2r) \neq \emptyset$ for all $X,Y \in \Phi^{(1)} \cap R_e$ such that $B_X(2r_1) \cap B_Y(2r_1) \neq \emptyset.$  
Let $B_e := \1 \{A_eC_e =1\}$ (see Figure~\ref{fig:1}).
Declare an edge $e \in l\mE \, ^2$ to be open if $B_e = 1$. We first show that for $\lam > \lam_c(2r)$ there exists a $\mu, l$ such that  $l\mL^2$ percolates (Step 1). The next step is to show that this implies percolation in the continuum model $G(\lam,\mu,r)$ (Step 2).

{\sc Step 1:} The random variables $\{B_e\}_{e \in l\mE^2}$ are $1$-dependent, that is, $B_e$'s indexed by two non-adjacent edges (edges that do not share a common vertex) are independent. Hence, given edges $e_1,\ldots,e_n \in l\mE^2$, there exists $\{k_j\}_{j=1}^m \subset \{1,\ldots,n\}$ with $m \geq n/4$ such that $\{B_{e_{k_j}}\}_{1 \leq j \leq m}$ are i.i.d. Bernoulli random variables. Hence,
\begin{equation}
\label{eqn:prob_peierl}
 \pr{B_{e_i} = 0, \ 1 \leq i \leq n} \leq \ \pr{B_{e_{k_j}} = 0, 1 \leq j \leq m} \leq  \pr{B_e = 0}^{n/4}.
\end{equation}
We need to show that for a given $\epsilon > 0$ there exists $l,\mu$, for which $\pr{B_e = 0} < \epsilon$ for any $e \in l\mE^2$.
Fix an edge $e$. Observe that
\begin{eqnarray}
 \pr{B_e = 0} & = & \pr{A_e = 0} + \pr{B_e=0|A_e=1}\pr{A_e=1} \nonumber \\
 & \leq & \pr{A_e=0} + \pr{B_e=0|A_e=1}.
\label{eqn:prob_Be}
\end{eqnarray}
Since $\lam > \lam_c(2r_1),$ $\tG(\lam, 2r_1)$ percolates. Hence by \cite[Corollary 4.1]{Meester96}, we can and do choose a $l$ large enough so that
\begin{equation}
\label{eqn:prob_A}
\pr{A_e = 0} < \frac{\epsilon}{2}.
\end{equation}
Now consider the second term on the right in (\ref{eqn:prob_Be}). Given $A_e = 1$, there exist crossings as specified in the definition of $A_e$ in $\tG(\lam,2r_1)$. Draw balls of radius $2r (> 2r_1)$ around each vertex. 
Any two vertices that share an edge in  $\tG(\lam,2r_1)$ are centered at a distance of at most $4r_1$. The width of the lens of intersection of two balls of radius $2r$ whose centers are at most $4r_1 (< 4r)$ apart is bounded below by a constant, say $b(r,r_1) > 0$. Hence if we cover $R_e$ with disjoint squares of diagonal-length $b(r,r_1)/3$, then every lens of intersection will contain at least one such square. Let $S_j, j = 1,\ldots,N(b),$ be the disjoint squares of diagonal-length $b(r,r_1)/3$ that cover $R_e$. Note that
\begin{eqnarray*}
\pr{B_e = 1|A_e = 1} & \geq & \pr{\Phi^{(2)} \cap S_j \neq \emptyset, 1 \leq j \leq N(b)} \\
& = &(1-\exp(-\frac{\mu b(r,r_1)^2}{18}))^{N(b)} \to 1, \ \mbox{as $\mu \to \infty$}.
\end{eqnarray*}
Thus for the choice of $l$ satisfying (\ref{eqn:prob_A}), we can choose a $\mu$ large enough such that
\begin{equation}
\label{eqn:prob_B}
\pr{B_e=0|A_e=1} < \frac{\epsilon}{2}.
\end{equation}
From (\ref{eqn:prob_Be}) - (\ref{eqn:prob_B}), we get $\pr{B_e = 0} < \epsilon$. Hence given any $\epsilon > 0$, it follows from (\ref{eqn:prob_peierl}) that there exists $l, \mu$ large enough so that $\pr{B_{e_i}=0 , 1 \leq i \leq n} \leq \epsilon^{n/4}$. That $l\mL^2$ percolates now follows from a standard Peierl's argument as in \cite[pp. 17, 18]{Grimmett99}.

{\sc Step 2:} By Step 1, choose $l,\mu$ so that $l \mL ^2$ percolates. Consider
any infinite component in $l \mL ^2$. Let $e,f$ be any two adjacent edges
in the infinite component. In particular $B_e = B_f =1$. This has two implications, the first one being that there exists crossings $I_e$ and $I_f$ of $R_e$ and $R_f$ respectively in $\tG(\lam,2r_1)$. Since $e,f$ are adjacent, $R_{ei} = R_{fj}$ for some $i,j \in \{1,3\}$. Hence there exists a crossing $J$ of $R_{ei}$ in $\tG(\lam, 2r_1)$ that intersects both $I_e$ and $I_f$. 
Draw balls of radius $2r$ around each vertex of the crossings $J,I_e,I_f$.
The second implication is that every pairwise intersection of  these balls will contain at least one point of $\Phi^{(2)}$. This implies that $I_e$ and $I_f$ belong to the same $AB$ component in $G(\lam,\mu,r)$. Therefore $G(\lam,\mu,r)$ percolates when $l\mL^2$ does. \qed

{\bf Proof of Proposition \ref{prop:upper_bd}.} Recall Definition~\ref{def:p_c_a}. For $d=2$, let $\mathbb{T}$ be the triangular
lattice with edge length $r_0/2$, and let $Q_z$ be the flower centred at $z \in \mathbb{T}$ as shown in Figure~\ref{fig:flower}. For $d \geq 3$, let $\mZ^{*d}_{r_0} := (\frac{r_0}{2\sqrt{d}} \mZ^d, \{<z,z_1> \in (\frac{r_0}{2\sqrt{d}} \mZ^d) \times (\frac{r_0}{2\sqrt{d}} \mZ^d) : \|z-z_1\| = \frac{r_0}{2\sqrt{d}}\})$ and $Q_z$ be the cube of side-length $\frac{r_0}{2\sqrt{d}}$ centred at $z \in \mathbb{Z}^{*d}_{r_0}$. Note that the flowers or cubes are disjoint. We declare $z$ open if $Q_z \cap \Phi^{(i)} \neq \emptyset, \; 1 \leq i \leq k$. This is clearly a Bernoulli site percolation model on $\mathbb{T} \;\; (d=2)$ or $\mathbb{Z}^{*d}_{r_0} \;\; (d \geq 3)$ with probability $\prod_{i=1}^k (1-e^{-\lam_i a(d,r_0)})$ of $z$ being open. By hypothesis, $\prod_{i=1}^k (1-e^{-\lam_i a(d,r_0)}) > p_c(d)$, the critical probability for Bernoulli site percolation on $\mathbb{T} \;\; (d=2)$ or $\mathbb{Z}^{*d}_{r_0} \;\; (d \geq 3)$ and hence the corresponding graphs percolate. Let $<z_1,z_2,...>$ denote an infinite percolating path in $\mathbb{T} \;\; (d=2)$ or $\mathbb{Z}^{*d}_{r_0} \;\; (d \geq 3)$. Since it is a percolating path, almost surely, for all $i \geq 1,$ and every $j = 1,2, \ldots , k,$ $\Phi^{(j)}(Q_{z_i}) > 0$, that is, each (flower or cube) $Q_{z_i}$ contains a point of each of $\Phi^{(1)},\ldots,\Phi^{(k)}$. Hence almost surely, for every word $\{w(i)\}_{i \geq 1}$ we can find a sequence $\{X_i\}_{i \geq 1}$ such that for all $i \geq 1, \; X_{i} \in \Phi^{(w(i))} \cap Q_{z_i}$. Further, $|X_{i} - X_{i+1}| \leq r_0 \leq r_{w(i)} + r_{w(i+1)}.$  Thus, almost surely, every word occurs. \qed

{\bf Proof of Corollary \ref{cor:upper_bound}.}   
Apply Proposition \ref{prop:upper_bd} with $k=2$, $\lam_1 = \lam$, $\lam_2 = \mu$, $r_1 = r_2 = r$, and so $r_0 = 2r$. It follows that almost surely, every word occurs provided $(1-e^{-\lam a(d,2r)})(1-e^{-\mu a(d,2r)}) > p_c(d)$. In particular, under the above condition, almost surely, the word $(1,2,1,2,\ldots)$ occurs. This implies that there is a sequence $\{X_i\}_{i \geq 1}$ such that $X_{2j-1} \in \Phi^{(1)}$, $X_{2j} \in \Phi^{(2)}$, and $|X_{2j} - X_{2j-1}| \leq 2r,$ for all $j \geq 1$. But this is equivalent to percolation in $G(\lam,\mu,r)$. This proves the corollary once we note that there exists a $\mu < \infty$ satisfying the condition above only if $(1-e^{-\lam a(d,2r)}) > p_c(d)$, or equivalently $a(d,2r)\lam > \log(\frac{1}{1 - p_c(d)})$ and the least such $\mu$ is given in the RHS of (\ref{eqn:mu_c_ub}). \qed

{\bf Proof of Corollary \ref{cor:AB_perc}.} By the given condition $(1-e^{-\lam a(d,r)/2}) > \sqrt{p_c(d)}$, and continuity, there exists an $\epsilon > 0$ such that for all $p \in (1/2-\epsilon,1/2+\epsilon)$, we have $(1-e^{-\lam pa(d,r)}) > \sqrt{p_c(d)}$. Thus for all $p \in (1/2-\epsilon,1/2+\epsilon)$, we get that $(1-e^{-\lam pa(d,r)})(1-e^{-\lam (1-p)a(d,r)}) > p_c(d).$ Hence by invoking Proposition \ref{prop:upper_bd} as in the proof of Corollary \ref{cor:upper_bound} with $\lam_1 = \lam p, \lam_2 = \lam (1-p), r_1 = r_2 = r$, we get that $\widehat{\theta}(\lam,p,r) = 1.$ \qed

\section{Proofs for Section~\ref{sec:rgg_model}}
\label{sec:rgg_proofs}
For the lower bound of connectivity threshold, the following result analogous to \cite[Theorem 7.1]{Penrose03} will suffice.
\begin{prop}
\label{prop:liminf_M_n}
Let $M_n$ and $r_n(c)$ be as defined in Definition \ref{defn:iso_lnnr} and (\ref{eqn:r_n(c,1)}) respectively. Then for any $c > 0$ and $a < 1$, $\pr{\mbox{$M_n \leq a^{\frac{1}{d}}r_n(c)$ i.o.}} = 0,$ where i.o. stands for infinitely often.
\end{prop}
\begin{proof}
For $a < 1$, set $r_n = a^{\frac{1}{d}}r_n(c)$ and choose a $\epsilon > 0$  such that
$$ \ep^{\frac{1}{d}} + a^{\frac{1}{d}} < (1- \ep)^{\frac{1}{d}}.$$
For $x \in U$, define the events :
$$ A_n(x) := \{\Pdm(B_x((1- \ep)^{\frac{1}{d}}r_n(c))) = 0\} \cap \{\Pdl(B_x(\ep^{\frac{1}{d}}r_n(c))) \geq 1\}.$$
Choose points $x_1^n,\ldots,x_{\sg_n}^n$ in $U$ of maximal cardinality such that the balls $B_{x_i^n}((1- \ep)^{\frac{1}{d}}r_n(c)),$ $1 \leq i \leq \sigma_n$ are disjoint. By \cite[Lemma 5.2]{Penrose03}, we can choose a constant $0 < \kappa < 1$ such that for all large enough $n$
\begin{equation}
\label{eqn:sigma_n}
 \sigma_n > \kappa \frac{n}{\log n}  
\end{equation}
If $A_n(x)$ occurs for some $x \in U$, then there exists a point $X \in \Pdl \cap B_x(\ep^{\frac{1}{d}}r_n(c))$ such that for all $Y \in \Pdm$
$$ d(X,Y) \geq \left((1-\epsilon)^{\frac{1}{d}} - \epsilon^{\frac{1}{d}}\right) r_n(c) > a^{\frac{1}{d}} r_n(c),$$
by the choice of $\epsilon$. It follows that $X$ is an isolated node in $G_n(cn,r_n)$ or equivalently, $M_n > r_n.$ Therefore,
\begin{equation}
\label{eqn:subset} 
\{M_n \leq r_n \} \subset (\cup_{i =1}^{\sigma_n}A_n(x_i))^c.
\end{equation}
For all $n$ large enough we have
$$ \pr{\Pdl(B_x(\ep^{\frac{1}{d}}r_n(c))) \geq 1} = 1  - n^{-\frac{\ep}{c}} \geq \kappa,$$
and 
$$\pr{\Pdm(B_x((1- \ep)^{\frac{1}{d}}r_n(c))) = 0} =  n^{\ep-1}.$$
Since $\Pdl$ and $\Pdm$ are independent, we get that for all large enough $n$,
$$ \pr{A_n(x_i^n)} \geq \kappa  n^{\ep-1}, \qquad 1 \leq i \leq \sigma_n.$$
By the above estimate, the independence of events $A_n(x_i^n), 1 \leq i \leq \sigma_n$, (\ref{eqn:sigma_n}) and the 
inequality $1 - t \leq e^{-t}$, we get that for all large enough $n$, 
$$\pr{(\cup_{x \in \mR^d}A_n(x))^c} \leq  \pr{(\cup_{i=1}^{\sigma_n}A_n(x_i^n))^c} \leq \exp\{-\kappa \sigma_n n^{\ep-1}\} \leq \exp\{-\kappa^2 \frac{n^{\ep}}{\log n}\},$$
which is summable in $n$. It follows by the Borel-Cantelli lemma and (\ref{eqn:subset}) that for $a < 1$, with probability $1$, $M_n > r_n$ for all large enough $n$.  
\end{proof}

We now prove Theorem~\ref{thm:conn}. In the second part of this proof, we will couple our sequence of $AB$ random geometric graphs with a sequence of random geometric graphs. By a random geometric graph, we mean the graph $\uG_n(r)$ with vertex set $\Pdl$ and edge set $\{ \la X_i, X_j \ra: X_i,X_j \in \Pdl, d(X_i,X_j) \leq r \}$, where $d$ is the toroidal metric defined in (\ref{eqn:tm}). We will use the following well known result regarding strong connectivity in the
graphs $\uG_n(r)$.

\begin{thm} [Theorem 13.2, \cite{Penrose03}]
\label{thm:conn_rgg}
 For $R_n(A_0) = \left( \frac{A_0 \log n}{n \theta_d} \right)^{1/d}$, almost surely, the sequence of graphs $\uG_n(R_n(A_0))$ is connected eventually 
%(and hence {\it whp}) 
if and only if $A_0 > 1.$
\end{thm}
{\bf Proof of Thm \ref{thm:conn}.} Again, let $r_n = a^{\frac{1}{d}}r_n(c)$, where 
$r_n(c) = r_n(c,1)$ is as defined in (\ref{eqn:r_n(c,1)}). It is enough to show the following for $c > 0$ : 
\begin{eqnarray}
\label{eqn:liminf} \mbox{For } \, a < 1, \,& & \pr{G_n(cn,r_n) \mbox{ is connected i.o.}} \leq \pr{ \mbox{$M_n \leq r_n$ i.o.}} = 0 \, \, \, \mbox{and} \\
\label{eqn:limsup}  \mbox{for } \, a > \al(c), \, & & \pr{G_n(cn,r_n) \mbox{ is not connected i.o.}} = 0.
\end{eqnarray}
(\ref{eqn:liminf}) and (\ref{eqn:limsup}) give the lower and upper bounds in (\ref{eqn:connect}) respectively. (\ref{eqn:liminf}) follows immediately from Proposition \ref{prop:liminf_M_n}. 

We now prove (\ref{eqn:limsup}). Since $a > \al(c)$, by definition $a \eta(a, c) > 1$. By continuity, we
can, and do choose $A_0 > 1$ such that $a \eta(a,A_0 c) >1.$ Choose $\ep \in (0,1)$ 
so that
\begin{equation}
(1-\ep)^2 a \eta(a,A_0 c) > 1.
\label{eqn:cond_ep} 
\end{equation}
Let $R_n = R_n(A_0)$, where $R_n(A_0)$ is as defined in Theorem~\ref{thm:conn_rgg}. For each $X_i \in \Pdl$, define the event
$$A_i(n,m,r,R) := \{\mbox{$X_i$ connects to all points of $\Pdl \cap B_{X_i}(R)$ in $G_n(m,r)$}\},$$
and let 
$$B(n,m,r,R) = \cup_{X_i \in \Pdl} \; A_i(n,m,r,R)^c.$$ 
We want to show that the event that every point of $\Pdl$ is connected in $G_n(cn,r_n)$ to all points of $\Pdl$ 
that fall within a distance $R_n(A_0)$ for all $n$ large enough, happens almost surely, or equivalently,
$$ \pr{ B(n,cn,r_n,R_n) \;\; i.o.} = 0.$$
We will use a subsequence argument and the Borel-Cantelli Lemma to show this.
%To this end, we use a subsequence argument to show that if $a > \al(c)$, then we can find $A_0 > 1$ such that the 
%probability of the event described above is summable. (\ref{eqn:limsup}) then follows from Theorem~\ref{thm:conn_rgg} %and the Borel-Cantelli Lemma.
%
%first calculate the probability that every point of $\Pdl$ is connected to all %points of $\Pdl$ that fall within a distance $R_n(A_0)$ for some $A_0 > 1$ in %$G_n(cn,r_n)$. Using this bound and a subsequence argument we will prove %summability of this probability for any $a > a(c)$.
%This will imply that the graphs $G(n,R_n)$ are connected {\it whp} by Fact %\ref{fact:conn_rgg}, and so we have that $GR(n,cn,ar_n)$ is also connected {\it %whp}.
Observe that $B(n,m,r,R) \subset B(n_1,m_1,r_1,R_1)$, provided $n \leq n_1,m \geq m_1, r \geq r_1, R \leq R_1$.  Let $n_j = j^{b}$ for some integer $b > 0$ that will be chosen later. 
%Then, $r_{n_j}^d = \frac{a\log n_j}{cn_j\theta_d}$ and $R_{n_j}^d = %\frac{A_0\log n_j}{\theta_d n_j}$. 
Since $B(n,cn,r_n,R_n) \subset B(n_{j+1},cn_j,r_{n_{j+1}},R_{n_j})$, for $n_j \leq n \leq n_{j+1}$, 
\begin{equation}
\cup_{n = n_j}^{n_{j+1}} \; B(n,cn,r_n,R_n) \subset B(n_{j+1},cn_j,r_{n_{j+1}},R_{n_j}).
\label{eqn:B_n_subset} 
\end{equation}
Let $p_{j} = \pr{A_i(n_{j+1},cn_j,r_{n_{j+1}},R_{n_j})^c}$. Let $N_n = \Pdl([0,1]^2)$. From (\ref{eqn:B_n_subset}) 
and the union bound we get
\bea
\pr{\cup_{n=n_j}^{n_{j+1}}B(n,cn,r_n,R_n)} & \leq & \pr{B(n_{j+1},cn_j,r_{n_{j+1}},R_{n_j})} \no \\
& \leq &  \pr{\cup_{i=1}^{N_{n_{j+1}}} A_i(n_{j+1},cn_j,r_{n_{j+1}},R_{n_j})^c} \no \\
& \leq & \sum_{i=1}^{n_{j+1}+n_{j+1}^{\frac{3}{4}}} \pr{A_i(n_{j+1},cn_j,r_{n_{j+1}},R_{n_j})^c} + \pr{|N_{n_{j+1}}-n_{j+1}| > n_{j+1}^{\frac{3}{4}}} \no \\
& \leq & 2n_{j+1} \, p_{j} + \pr{|N_{n_{j+1}}-n_{j+1}| > n_{j+1}^{\frac{3}{4}}}. \label{eqn:prob_Ai}
\eea
\remove{Now, we shall focus on the first term. Let $\theta_d r_{n_{j+1}}^d \zeta(r_{n_{j+1}},R_{n_j}) = \inf \{ \|B_O(r_{n_{j+1}}) \cap B_{x}(r_{n_{j+1}})\| : x \in B_O(R_{n_j}) \}$ and $x_n$ be the point at which the infimum is attained. Thus, $x_n \in \partial B_O(R_{n_j})$. From Lemma \ref{lem:area_lens},   $\zeta(r_{n_{j+1}},R_{n_j}) = \eta(r_{n_{j+1}},R_{n_j})$ when $d=2,$ and $\zeta(r_{n_{j+1}},R_{n_j}) \geq \eta(r_{n_{j+1}},R_{n_j})$ when $d \geq 3$.} 

We now estimate $p_{j}$. Let $e_1 =(1,0, \ldots ,0) \in \mR^d$. Conditioning on the number of points of 
$\P_{n_{j+1}}$ in $B_O(R_{n_j})$ and then using the Boole's inequality, we get
\bea
p_{j} & \leq &  \sum_{k=0}^{\infty} \frac{(n_{j+1}\theta_d R_{n_j}^d)^ke^{-n_{j+1}\theta_d R_{n_j}^d}}{k!} \frac{k}{\theta_d R_{n_j}^d} \int_{B_O(R_{n_j})} e^{-cn_j\|B_O(r_{n_{j+1}})\cap B_x(r_{n_{j+1}})\|} dx \no \\
& \leq & \sum_{k=0}^{\infty} \frac{(n_{j+1}\theta_d R_{n_j}^d)^ke^{-n_{j+1}\theta_d R_{n_j}^d}}{k!} \frac{k}{\theta_d R_{n_j}^d} \int_{B_O(R_{n_j})} e^{-cn_j\|B_O(r_{n_{j+1}})\cap B_{R_{n_j}e_1}(r_{n_{j+1}})\|} dx  \no \\
& = & n_{j+1}\theta_d R_{n_j}^d e^{-c n_j \theta_d r_{n_{j+1}}^d \eta(r_{n_{j+1}}^d,R_{n_j}^d)},
\label{eqn:p_j_int}
\eea
where $\eta(\cdot, \cdot)$ is as defined in (\ref{eqn:defn_eta}). Since
\[ \frac{R_{n_j}}{r_{n_{j+1}}} = \left( \frac{A_0 \log n_j}{\theta_d n_j}
\frac{c n_{j+1} \theta_d}{a \log n_{j+1}} \right)^{\frac{1}{d}}
\to \left( \frac{A_0 c}{a} \right)^{\frac{1}{d}}, \]
by the continuity of $\eta(.,.)$ (this follows from (\ref{eqn:exp_eta})), we have
\begin{equation} \eta(r_{n_{j+1}}^d,R_{n_j}^d) \geq (1-\ep) \; \eta(a,A_0 c),
\label{eqn:asy_bound_L}
\end{equation}
for all sufficiently large $j$.
%, where $\eta$ is as defined in (\ref{eqn:defn_eta}). 
For all $j$ sufficiently large, we also have
$(\frac{j}{j+1})^{b} \geq (1-\ep)$. Using (\ref{eqn:asy_bound_L}) and
simplifying by substituting for $R_{n_j}$ and $r_{n_{j+1}}$ in (\ref{eqn:p_j_int}),
for all sufficiently large $j$, we have
\beas
p_{j} & \leq & \frac{(j+1)^{b} \, A_0 \, b \,\log j}{j^{b}} e^{-\frac{j^b}{(j+1)^b} \, (1-\ep) \, \eta(a, A_0 c) \, a \, b \, \log(j+1)} \\
& \leq &   \frac{ A_0 \, b\, \log j}{(1-\ep)} e^{- (1-\ep)^2 \, \eta(a,A_0c) \, a \, b \, \log(j+1)} \\
& = &  \frac{ A_0 \, b \,\log j}{(1-\ep) (j+1)^{(1-\ep)^2 \, \eta(a,A_0c) \, a \, b}}.
\eeas
Hence 
\begin{equation}
n_{j+1} \, p_{j} \leq \frac{ A_0 \, b \,\log j}{(1-\ep) (j+1)^{((1-\ep)^2 \, \eta(a,A_0c) \, a \, - 1)b}}.
\label{eqn:bound_n_p_n}
\end{equation} 
Using (\ref{eqn:cond_ep}), we can choose $b$ large enough so that
$((1-\ep)^2 \, \eta(a,A_0c) \, a \, - 1)b > 1.$ It then follows from
(\ref{eqn:bound_n_p_n}) that the first term on the right in (\ref{eqn:prob_Ai}) is summable in $j$. From \cite[Lemma 1.4]{Penrose03}, the second term on the right in (\ref{eqn:prob_Ai}) is also summable. 
\remove{
Choose $\alpha > 0$ such that $(\gamma (\eta(a,A_0c)-\epsilon)a - 1)\al - 1 = a_0 > 0 $. Then $\frac{n_{j+1}}{{(j+1)^{\gamma (\eta(a,A_0c)-\epsilon)a \al}}} = (j+1)^{-a_0-1}$ and so
$$ \sum_{j \geq 0} \pr{\cup_{k=n_j}^{n_{j+1}}B(n_{k},cn_k,r_{n_k},R_{n_k})} \leq \sum_j \left( \frac{2\al A_0\log j}{\gamma (j+1)^{a_0+1}} + \pr{|N_{n_{j+1}}-n_{j+1}| > n_{j+1}^{\frac{3}{4}}} \right) < \infty.$$}

Hence by the Borel-Cantelli Lemma, almost surely, only finitely many of the events 
$$\cup_{n=n_j}^{n_{j+1}}B(n,cn,r_n,R_n)$$ 
occur, and hence only finitely many of the events $B(n,cn,r_n,R_n)$ occur. This implies that almost surely, every vertex in $G_n(cn,r_n)$ is connected to 
every other vertex that is within a distance $R_n(A_0)$ from it, for all large $n$. Since $A_0 > 1,$ it follows from Theorem~\ref{thm:conn_rgg} that almost surely, $G_n(cn,r_n)$ is connected eventually.
This proves (\ref{eqn:limsup}). \qed

Towards a proof of Lemma \ref{lem:cgs_exp_isoln}, we first derive a vacancy estimate similar to \cite[Theorem 3.11]{Hall88}. For any locally finite point process $\mathcal{X} \subset U,$ the coverage process is defined as 
\begin{equation}
\cC(\mathcal{X},r) := \bigcup_{X_i \in \mathcal{X}}B_{X_i}(r),
\label{eqn:cov}
\end{equation}
and we abbreviate $\cC(\Pdl,r)$ by $\cC(n,r)$. Recall that for any $A \subset \mR^d$, we write $\mathcal{X}(A)$ to be the number of points of $\mathcal{X}$ that lie in the set $A$. 
\begin{lem}
\label{lem:vacancy_estimate}
For $d =2$ and $0 < r < \frac{1}{2}$, define $V(r) := 1 - \frac{\|B_O(r) \cap \cC(n,r)\|}{\pi r^2}$, the normalised vacancy in the $r$-ball. Then $$ \pr{V(r) > 0} \leq (1 + n\pi r^2 + 4(n\pi r^2)^2) \; \exp( -n \pi r^2).$$
\end{lem}

{\bf Proof of Lemma~\ref{lem:vacancy_estimate}.} 
Write $\pr{V(r) > 0} \leq p_1 + p_2 + p_3$, where
\begin{eqnarray*}
p_1 & = & \pr{\Pdl(B_O(r)) = 0} = \exp(-n\pi r^2), \\
p_2 & = & \pr{\Pdl(B_O(r)) = 1} = n\pi r^2\exp(-n\pi r^2 ), \\
p_3 & = & \pr{\Pdl(B_O(r)) > 1, V(r) > 0}. 
\end{eqnarray*}
We shall now upper bound $p_3$ to complete the proof. A \emph{crossing} is defined as a point of intersection of the boundaries of two balls (all the balls mentioned in this proof are assumed to have a radius $r$) centred at points of $\Pdl$. A crossing is said to be \emph{covered} if it lies in the interior of another ball centred at a point of $\Pdl$, else it is said to be \emph{uncovered}. If there is more than one point of $\Pdl$ in $B_O(r)$, then there exists at least one crossing in $U$. If $V(r) > 0$ and there exists more than one ball centred at a point of $\Pdl$ in $B_O(r)$, then there exists at least one such ball with two uncovered crossings on its boundary. Denoting the number of uncovered crossings by $M$, we have that $$p_3 \leq \pr{M \geq 2} \leq \frac{\EXP{M}}{2}.$$ 

Note that balls centred at distinct points can have at most $2$ crossings and almost surely, all the points of $\Pdl$ are distinct. Thus, given a ball, the number of crossings on the boundary of the ball is twice the number of balls centred at a distance within $2r$. This number has expectation $2\int_0^{2r} 2n\pi x \ dx = 8 n\pi r^2$, where $2n\pi x \ dx$ is the expected number of balls whose centers lie between $x$ and $x+ \ dx$ of the center of the given ball. Thus, 
$$\EXP{M} = \EXP{\Pdl(B_O(r))}8n\pi r^2 \pr{\mbox{a crossing is uncovered}} = 8 (n\pi r^2)^2 \exp(-n\pi r^2). \qed $$
\remove{ 
\begin{lem}
\label{lem:area_lens}
For any $r > 0$ and $x \in \mR^d$ with $0 \leq R = |x| \leq 2r$, 
define $L(r,R) := \|B_O(r) \cap B_x(r)\|$. Then 
%$L(r,R) = h(r,R)$ for $d=2$ and $L(r,R) \geq h(r,R)$ for $d \geq 3$, 
%where 
%
\begin{eqnarray}
L(r,R) & = & \left( 2 \phi\left( \frac{R}{2r} \right) - 
\sin \left( 2\phi \left( \frac{R}{2r} \right) \right) \right) r^2, \qquad \mbox{if } d=2,  \nonumber \\
L(r,R) & \geq &  \theta_d \left(r - \frac{R}{2} \right)^d, \qquad  \mbox{if } d\geq 2,
\label{eqn:defn_L}
\end{eqnarray}
where $\phi(a) = \arccos(a)$. 
\end{lem}
{\bf Proof of Lemma \ref{lem:area_lens}.}
\begin{figure}[h!t]
  \centerline{\includegraphics[scale=0.5]{fig2.eps}}
  \caption{$|x| = R, \phi = \phi(\frac{r}{2R})$ and $L(r,R)$ is the area of the lens of intersection, the shaded region.}
  \label{fig:3}
\end{figure}

\selfnote{REDRAW FIGURE }

%We sketch the derivation here first for $d=2$ and then for $d \geq 3$. We %sketch them seperately as we obtain tighter bounds for $d=2$ than in higher %dimensions. 
Let $d = 2$. From Figure \ref{fig:3},
%Refer to Figure \ref{fig:3} for this derivation when $d=2$. Let $L(r,R) = %\|B_O(r) \cap B_x(r)\|, \|x\| = R$ be the area of lens of intersection of %circles (or spheres) of radius $r$ centered $R ( < 2r)$ apart. 
it is clear that $L(r,R)$ is cut into two equal halves by the line $PQ$ and the area of each of those halves is the area enlosed between the chord $PQ$ in the ball $B_O(r)$ and its circumference. The area of the angular sector $OPQ$ (with $PQ$ considered as the arc along the boundary of the ball) is $\phi\left( \frac{R}{2r} \right) r^2$. The area of the triangle $OPQ$ is 
\[ r \sin \left( \phi \left( 
\frac{R}{2r} \right) \right) \times r \cos \left( \phi \left( 
\frac{R}{2r} \right) \right) = \frac{r^2}{2}\sin \left( 2\phi \left( 
\frac{R}{2r} \right) \right).\] 
Hence $L(r,R) = \left(2\phi\left( 
\frac{R}{2r} \right) - \sin\left(2\phi\left( 
\frac{R}{2r} \right)\right)\right)\,r^2$. 

Consider the case $d \geq 2$. The width of the lens of intersection 
of the balls $B_O(r)$ and $B_x(r)$ is $2r-R$. Thus the lens of intersection contains a ball of diameter $2r-R$. Hence the volume of such a ball, $\theta_d(r-\frac{R}{2})^d$, is a lower bound for $L(r,R)$. \qed
}

{\bf Proof of Lemma~\ref{lem:cgs_exp_isoln}.} We first prove the second part of the Lemma which is easier.

{\sc(2).} Let $\widehat{W}_n(r)$ be the number of $\Pdl$ nodes for which there 
is no other $\Pdl$ node within distance $r$. Note that $\widehat{W}_n(2r) \leq W_n(r).$ By this inequality and the Palm calculus, we get
\begin{eqnarray*}
\EXP{ W_n(r_n(c,\beta))} & \geq & \EXP{\widehat{W}_n(2r_n(c,\beta))} \\ 
& = & n \int_{U} \pr{\Pdl(B_x(2r_n(c,\beta))) = 0} \, dx \\
 & =  & n \, \exp( - 2^d n \theta_d r_n^d(c,\beta) ) 
\; = \; n \, \exp\left( - \frac{2^d}{c} \log (\frac{n}{\beta}) \right) \;  \to \; \infty, 
\end{eqnarray*}
as $n \to \infty$ since $c > 2^d$.

{\sc (1).} We prove the cases $d = 2$ and $d \geq 3$ separately. 

Let $d \geq 3$ and fix $c < 1$. Define $\widetilde{W}_n(c,r)$ to be the number of $\Pdl$ nodes for which there is no $\Pdc$ nodes within distance $r$ and $\overline{W}_n(c,r)$ be the number of $\Pdc$ nodes with only one $\Pdl$ node within distance $r$. Note that 
\begin{equation} \widetilde{W}_n(c,r) \leq W_n(r) \leq \widetilde{W}_n(c,r) + \overline{W}_n(c,r).
\label{eqn:lb_ub_iso}
\end{equation} 
By Palm calculus for Poisson point processes, we have
\begin{eqnarray}
\EXP{\widetilde{W}_n(c,r_n(c,\beta))} & = & n \int_{U} \pr{\Pdc(B_x(r_n(c,\beta))) = 0} dx \no \\
& = & n \exp (-cn\theta_d r^d_n(c,\beta)) \; = \; \beta, \label{eqn:cgs_bound1_iso}\\
\EXP{\overline{W}_n(c,r_n(c,\beta))} & = & cn\int_{U} \pr{\Pdl(B_x(r_n(c,\beta))) = 1} dx \no \\
 & = &  c \, n \exp (-n \theta_d r^d_n(c,\beta)) \, n \, \theta_d \, r^d_n(c,\beta)  \to 0, 
\label{eqn:cgs_bound2_iso}
\end{eqnarray}
%
%Now it is a matter of substitution to see that 
%$ \EXP{\widetilde{W}_n(c,r_n(c,\beta))} = \beta$, and
%$\EXP{W_n^1(c,r_n(c,\beta))} \to 0,$ 
since $c < 1$. It follows from (\ref{eqn:lb_ub_iso}), (\ref{eqn:cgs_bound1_iso}) and (\ref{eqn:cgs_bound2_iso}) that $\EXP{W_n(r_n(c,\beta))} \to \beta$, as $n \to \infty$, if $d \geq 3$ and $c < c_0 = 1$.

Now let $d=2$, fix $c < c_0$, where $c_0$ is as defined in 
(\ref{eqn:defn_c_0}) and let $n$ be large enough such that $r_n(c,\beta) < \frac{1}{2}$. For any $X \in \Pdl$, using
(\ref{eqn:cov}), the degree of $X$ in the graph $G_n(cn,r)$ can be written as 
\[ deg_n(cn,X) := \sum_{X_j \in \Pdl}\1 \{<X_j,X> \in E_n(cn,r) \} 
= \Pdl ( \cC((\Pdc \cap B_X(r)), r) \setminus \{X\} ).  \] 
%
%where the notation in the second equality above is from (\ref{eqn:cov}). 
Since 
\begin{equation}
\label{eqn:2rep_isol_nodes}
\{ \Pdl ( \cC((\Pdc \cap B_X(r)) ,r) \setminus \{X\})  = 0 \} = 
\{ \Pdc( B_{X}(r) \cap \cC(\Pdl \setminus \{ X \},r)) = 0 \},
\end{equation} 
we have
\begin{equation}
\label{eqn:rep_isol_nodes}
W_n(r) = \sum_{X_i \in \Pdl} \1 \{ deg_n(cn,X_i) = 0 \} = 
\sum_{X_i \in \Pdl} \1 \{\Pdc(B_{X_i}(r) \cap 
\cC(\Pdl \setminus \{ X_i \},r)) = 0 \}.
\end{equation} 
By Palm calculus for Poisson point processes (and the metric being toroidal) we have,
\be
\label{eqn:exp_isoln_palm}
 \EXP{W_n(r)} =  n \int_{U}\EXP{\1\{deg_n(cn,x) = 0\}}dx = n \pr{\Pdc(B_O(r) \cap \cC(n,r)) = 0},
\ee
where $\cC(n,r) = \cC(\Pdl,r)$. For any bounded random closed set $F$, conditioning on $F$ and then taking expectation, we have 
\begin{equation} \pr{\Pdc(F) = 0} = \EXP{\exp (-cn \|F\|)}. 
\label{eqn:exp_cov_rand_set}
\end{equation} 
Thus from (\ref{eqn:exp_isoln_palm}), (\ref{eqn:exp_cov_rand_set}) we get 
\begin{equation}
\EXP{W_n(r)} = n \ \EXP{\exp (-cn \|B_O(r) \cap \cC(n,r)\|)} = n \ \EXP{ \exp(-cn\pi r^2(1-V(r)))},
\label{eqn:int_W_n}
\end{equation} 
where $V(r)$ is as defined in Lemma~\ref{lem:vacancy_estimate}. Let $\eta(c) = \eta(1,c)$ be as defined in (\ref{eqn:defn_eta}) and $e_1 = (1,0)$. Since $\frac{r_n(1,\beta)}{r_n(c,\beta)} = c^{\frac{1}{2}}$, by (\ref{eqn:exp_eta}) we have
\begin{equation}  
\frac{\|B_O(r_n(c,\beta)) \cap B_{r_n(1,\beta)e_1}(r_n(c,\beta))\|}
{\pi r_n(c,\beta)^2} =  \eta(c). 
%& = & \pi^{-1} \left( 2 \phi\left( \frac{r_n(1,\beta)}{2r_n(c,\beta)} \right) - 
%\sin \left( 2\phi \left( \frac{r_n(1,\beta)}{2r_n(c,\beta)} \right) \right) \right) \\
\end{equation}
Given $c < c_0$, by continuity, we can choose an $\ep \in (0, 1)$, such that 
%
%\[A(c,\ep) = \frac{\|B_O(r_n(c,\beta)) \cap B_{r_n(1 - %\ep,\beta)e_1}(r_n(c,\beta))\|}{\pi r_n(c,\beta)^2}  \] 
%
%satisfies 
%
\begin{equation}
\eta_*(c,\ep) = \frac{\|B_O(r_n(c,\beta)) \cap B_{r_n(1 - \ep,\beta)e_1}(r_n(c,\beta))\|}{\pi r_n(c,\beta)^2} 
\, \, \, \, \, \, \, \, \mbox{satisfies} \, \, \, \, \, \, \, \, \eta_*(c,\ep) + \frac{1}{c} > 1.
\label{eqn:cond_A_c}
\end{equation} 
Let $N_n = \Pdl(B_O(r_n(1 - \ep,\beta)))$. Thus, we have
\begin{eqnarray}
\EXP{W_n(r_n(c,\beta))} & = & 
n \, \EXP{e^{-cn\pi r_n^2(c,\beta)(1-V(r_n(c,\beta)))}\1\{V(r_n(c,\beta))=0\}} \nonumber \\ 
&  & + n \, \EXP{e^{-cn\pi r_n^2(c,\beta)(1-V(r_n(c,\beta)))}\1\{V(r_n(c,\beta)) > 0, N_n = 0\}} 
\nonumber\\
& & + n \, \EXP{e^{-cn\pi r_n^2(c,\beta)(1-V(r_n(c,\beta)))}\1\{V(r_n(c,\beta)) > 0, N_n > 0\}}. 
\label{eqn:E_W_n}
\end{eqnarray}
Consider the first term in (\ref{eqn:E_W_n}). From Lemma \ref{lem:vacancy_estimate}, we obtain the bound,
\begin{equation}
\pr{V(r_n(c,\beta)) > 0} \leq D (1+\log n + 4(\log n)^2) n^{-\frac{1}{c}},
\label{eqn:bound_vacancy1}
\end{equation} 
for some constant $D$. Hence,
\begin{eqnarray}
n \, \EXP{e^{-cn\pi r_n^2(c,\beta)(1-V(r_n(c,\beta)))}\1\{V(r_n(c,\beta))=0\}}  
& = & n \exp(-cn\pi r_n(c,\beta)^2)\pr{V(r_n(c,\beta))=0}) \nonumber \\
& = & \beta \, \pr{V(r_n(c,\beta))=0} \to \beta,
\label{eqn:E_W_n_first}
\end{eqnarray}
as $n \to \infty$. The second term in (\ref{eqn:E_W_n}) is bounded by
\begin{equation}
n \, \pr{ N_n = 0} = n \exp(-n\pi r_n(1-\ep,\beta)^2) = n^{1-\frac{1}{1-\ep}}\beta^{\frac{1}{1-\ep}} \to 0,
\label{eqn:E_W_n_second}
\end{equation}
as $n \to \infty$. We will now show that the third term in (\ref{eqn:E_W_n}) converges to $0$. On the event $\{ N_n > 0 \}$, we have 
\begin{equation}
1 - V(r_n(c,\beta)) > \eta_*(c,\ep).
\label{eqn:bound_vacancy2}
\end{equation} 
Using (\ref{eqn:bound_vacancy2}) first and then (\ref{eqn:bound_vacancy1}), the third term in (\ref{eqn:E_W_n}) 
can be bounded by
\begin{eqnarray}
n e^{-cn \pi r_n(c,\beta)^2 \eta_*(c,\ep)}\pr{V(r_n(c,\beta)) > 0, N_n > 0}
& \leq &  n^{1-\eta_*(c,\ep)} \beta^{\eta_*(c,\ep)} \pr{V(r_n(c,\beta)) > 0} \nonumber \\
& \leq & D \; n^{1-\eta_*(c,\ep)-\frac{1}{c}} (1+\log n + 4(\log n)^2)\beta^{\eta_*(c,\ep)}
\nonumber \\
& \to & 0
\label{eqn:E_W_n_third}
\end{eqnarray}
as $n \to \infty$ by (\ref{eqn:cond_A_c}). 

It follows from (\ref{eqn:E_W_n}), (\ref{eqn:E_W_n_second}) and (\ref{eqn:E_W_n_third}) that
$\EXP{W_n(r_n(c,\beta))} \to \beta,$ as $n \to \infty.$ \qed

The \emph{total variation distance} between two integer valued random variables $\psi,\zeta$ is defined as
\begin{equation}
\label{eqn:tv}
d_{TV}(\psi,\zeta) = \sup_{A \subset \mZ} |\pr{\psi \in A} - \pr{\zeta \in A}|.
\end{equation}
%}
%Let $d_{TV}(\psi,\zeta)$ denote the \emph{total variation distance} between two %random variables $\psi,\zeta$.
The following estimate in the spirit of Theorem 6.7(\cite{Penrose03}) will be our main tool in proving Poisson convergence of $W_n(r_n(c,\beta))$. We denote the Palm version  $\Pdl \cup \{x\}$ of $\Pdl$ by $\Pdlx$.
\newpage
\begin{lem}
\label{lem:po_approx}
Let $0< r < 1$ and let $\cC(.\,,\,.)$ be the coverage process defined by (\ref{eqn:cov}). Define the integrals $I_{in}(r),$ $i=1,2$, and $n \geq 1$ by
\begin{eqnarray}
I_{1n}(r) & := & n^2 \int_{U}dx\int_{B_x(5r) \cap U} dy\ \pr{\Pdl(\cC(\Pdc \cap B_x(r),r)) = 0}\pr{\Pdl(\cC(\Pdc \cap B_y(r),r)) = 0}, \no \\
\label{eqn:int_po_approx} I_{2n}(r) & := & n^2 \int_{U} dx\int_{B_x(5r) \cap U}dy\ \pr{\Pdlx(\cC(\Pdc \cap B_y(r),r)) = 0 = \Pdly(\cC(\Pdc \cap B_x(r),r))}. 
\no \\
& & 
\end{eqnarray}
Then,
\begin{equation}
d_{TV}\left(W_n(r),Po(\EXP{W_n(r)})\right) \leq \min\left(3,\frac{1}{\EXP{W_n(r)}}\right)(I_{1n}(r) + I_{2n}(r)).
\label{eqn:poi_approx}
\end{equation} 

\end{lem}
{\bf Proof of Lemma \ref{lem:po_approx}.} 
The proof follows along the same lines as the proof of Theorem 6.7 (\cite{Penrose03}). For every $m \in \mN$, partition $U$ into disjoint cubes of side-length $m^{-1}$ and corners at $m^{-1}\mZ^d$. Let the cubes and their centres be denoted by $H_{m,1},H_{m,2},...$ and $a_{m,1},a_{m,2}...$ respectively. Define $I_m := \{i \in \mN \, : H_{m,i} \subset [0,1]^d\}$ and $E_m := \{<i,j> \, : i,j \in I_m, \; 0 < \|a_{m,i} - a_{m,j}\| < 5r \}$. The graph $G_m = (I_m,E_m)$ forms a dependency graph (see \cite[Chapter 2]{Penrose03}) for the random variables $\{\xi_{m,i}\}_{i \in I_m}$. The dependency neighbourhood of a vertex $i$ is $N_{m,i} = i \cup \{j : <i,j> \in E_m \}.$  Let
\[ \xi_{m,i} := \1\{\{\Pdl(H_{m,i})=1\} \cap \{\Pdl(\cC(\Pdc \cap B_{a_{m,i}}(r),r) \cap H_{m,i}^c) = 0\}\} .\]
$\xi_{m,i} = 1$ provided there is exactly one point of $\Pdl$ in the cube $H_{m,i}$ which is not connected to any other point of $\Pdl$ that falls 
outside $H_{m,i}$ in the graph $G_n(cn,r)$. Let $W_m = \sum_{i \in I_m}\xi_{m,i}$. Then almost surely, 
\begin{equation}
\label{eqn:lim_W^m}
W_n(r) = \lim_{m \to \infty} W_m.
\end{equation}
%
%We have made use of the representation of the degree of a node as noted in the 
%beginning of the Section (Para 1). 
Let $p_{m,i} = \EXP{\xi_{m,i}}$ and $p_{m,i,j} = \EXP{\xi_{m,i}\xi_{m,j}}$. The remaining part of the proof is based on the notion of dependency graphs and the Stein-Chen method. By  \cite[Theorem 2.1]{Penrose03}, we have
\begin{equation}
d_{TV}(W_m,Po(\EXP{W_m})) \leq min(3,\frac{1}{\EXP{W_m}})(b_1(m)+b_2(m)),
\label{eqn:poi_approx_m}
\end{equation}
where $b_1(m) = \sum_{i \in I_m} \sum_{j \in N_{m,i}} p_{m,i}p_{m,j}$ and $b_2(m) = \sum_{i \in I_m} \sum_{j \in N_{m,i}/\{i\}} p_{m,i,j}$. The result 
follows if we show that the expressions on the left and right
in (\ref{eqn:poi_approx_m}) converge to the left and right hand expressions 
respectively in (\ref{eqn:poi_approx}).

Let $w_m(x)= m^dp_{m,i}$ for $x \in H_{m,i}$. Then $\sum_{i \in I_m}p_{m,i} = \int_U w_m(x) \ dx$. Clearly,
\[ \lim_{m \to \infty} w_m(x) = n \pr{\Pdlx(\cC((\Pdc \cap B_x(r))/\{x\},r)) = 0} = n \pr{\Pdl(\cC(\Pdc \cap B_x(r),r)) = 0}.\] 
Since $w_m(x) \leq  m^d\pr{\Pdl(H_{m,i})=1} \leq n$,
\[ \lim_{m \rar \infty} \EXP{W_m} = n \int_U \pr{\Pdl(\cC(\Pdc \cap B_x(r),r)) = 0} \ dx = \EXP{W_n(r)}, \]
where the first equality is due to the dominated convergence theorem and the second follows from (\ref{eqn:2rep_isol_nodes}) - (\ref{eqn:exp_isoln_palm}). Similarly by letting $u_m(x,y) = m^{2d} p_{m,i}p_{m,j} \1\{[j \in N_{m,i}]\}$ and $v_m(x,y) = m^{2d}p_{m,i,j} \1\{[j \in N_{m,i}/\{i\}]\}$ for $x \in H_{m,i}, \ y \in H_{m,j}$, one can show that 
\begin{eqnarray*}
b_1(m) & = & \int_U u_m(x,y) \ dx \ dy \ \to \ I_{1n}(r), \\ 
b_2(m) & = & \int_U v_m(x,y) \ dx \ dy \ \to \ I_{2n}(r). \qed
\end{eqnarray*}  
{\bf Proof of Theorem \ref{thm:lnnd}. } (\ref{eqn:lnnd}) follows easily from (\ref{eqn:po_convg}) by noting that
\[ \pr{M_n \leq r} = \pr{W_n(r) = 0} . \] 
Hence, the proof is complete if we show (\ref{eqn:po_convg}) for which we
will use Lemma \ref{lem:po_approx}. Let $I_{in}(r_n(c,\beta)),$ $i=1,2$, be the integrals defined in (\ref{eqn:int_po_approx}) with $r$ taken to be $r_n(c,\beta)$ satisfying (\ref{ass:lt_regime}).
From Lemma \ref{lem:cgs_exp_isoln}, $\EXP{W_n(r_n(c,\beta))} \rar \beta$ as $n \to \infty.$ 
As convergence in \emph{total variation distance} implies convergence in distribution, by
Lemma \ref{lem:po_approx} and the conclusion in the last statement, it suffices to show that 
$I_{in}(r_n(c,\beta)) \to 0,$ as $n \to \infty$ for $i=1,2$. 

Using (\ref{eqn:exp_isoln_palm}) and Lemma \ref{lem:cgs_exp_isoln}, we get for some finite positive constant $C$ 
that
\[ I_{1n}(r_n(c,\beta))  =  \int_{U}dx\int_{B_x(5r_n(c,\beta)) \cap U} dy \ (\EXP{W_n(r_n(c,\beta))})^2 \leq C (5r_n(c,\beta))^d \rar 0, \qquad \mbox{ as } n \to \infty. \]
We now compute the integrand in the inner integral in $I_{2n}(r)$. Let $\Gamma(x,r) = \|B_O(r) \cap B_x(r)\|.$ For $x,y \in U$, using (\ref{eqn:exp_cov_rand_set}) we get
\begin{eqnarray}
\lefteqn{\pr{\{\Pdlx(\cC(\Pdc \cap B_y(r),r)) = 0\} \cap \{ \Pdly(\cC(\Pdc \cap B_x(r),r)) = 0\}}} \nonumber \\
& = & \pr{\Pdc(B_y(r) \cap (\cC(n,r) \cup B_x(r))) = 0,\Pdc(B_x(r) \cap (\cC(n,r) \cup B_y(r))) = 0} \nonumber \\
%& = & \pr{\Pdc \cap B_y(r) \cap \cC(n,r) = \emptyset, \Pdc \cap B_x(r) \cap %B_y(r) = \emptyset , \Pdc \cap B_x(r) \cap \cC(n,r) = \emptyset} \nonumber  \\
& \leq & \pr{\Pdc(B_y(r) \cap \cC(n,r)) = 0, \Pdc(B_x(r) \cap \cC(n,r)) = 0} \nonumber \\
%& = & \pr{\Pdc \cap (B_x(r) \setminus B_y(r) \cup B_y(r) \setminus B_x(r)) \cap \cC(n,r) = \emptyset,\Pdc \cap B_x(r) \cap B_y(r) = \emptyset} \\
%& \leq & \pr{\Pdc \cap (B_x(r) \setminus B_y(r) \cup B_y(r) \setminus B_x(r)) \cap \cC(n,r) = \emptyset,\Pdc \cap B_x(r) \cap B_y(r) \cap \cC(n,r) = \emptyset} \\
& = &  \pr{\Pdc((B_y(r) \setminus B_x(r)) \cap \cC(n,r)) = 0,\Pdc(B_x(r) \cap \cC(n,r)) = 0} \nonumber \\
& = &  \EXP{\exp( -cn \|(B_y(r) \setminus B_x(r)) \cap \cC(n,r)\|) \exp(-cn \|B_x(r) \cap \cC(n,r)\|) }.\label{eqn:integrand_I2}
\end{eqnarray}
%
%where the penultimate equality is derived as (\ref{eqn:int_W_n}) and the last %inequality follows from the following lower bound 
We can and do choose an $\eta >0$ so that for any $r>0$ and $|y-x| \leq 5r$
(see \cite[Eqn 8.21]{Penrose03}), we have 
\[ \|B_x(r) \setminus B_y(r)\| \  \geq \ \eta \ r^{d-1} \ |y-x|. \] 
Hence if $|y-x| \leq 5r$, the left hand
expression in (\ref{eqn:integrand_I2}) will be bounded above by
\[ \EXP{\exp\left(-cn\eta r^{d-1}|y-x| \frac{\|(B_y(r) \setminus B_x(r)) \cap \cC(n,r)\|}{\|B_y(r) \setminus B_x(r)\|}\right) \exp\left(-cn \|B_x(r) \cap \cC(n,r)\|
\right)}. \]
Using the above bound, we get
\begin{eqnarray*}
I_{2n}(r_n(c,\beta)) & \leq & \int_{B_O(5r_n^d(c,\beta)) \cap U} n^2 \mathsf{E} \Big(\exp\left(-cn \|B_O(r_n(c,\beta)) \cap \cC(n,r_n(c,\beta) )\| \right) \\
& & \exp\left( -cn\eta r_n(c,\beta)^{d-1}| y |\frac{\|(B_{y}(r_n(c,\beta)) \setminus B_O(r_n(c,\beta))) \cap \cC(n,r_n(c,\beta))\|}{\|B_{y}(r_n(c,\beta)) \setminus B_O(r_n(c,\beta))\|} \right) \Bigg) \ dy. 
\end{eqnarray*}
Making the change of variable $w = nr_n(c,\beta)^{d-1}y$ and using
(\ref{eqn:int_W_n}), we get
\begin{eqnarray*}
I_{2n}(r_n(c,\beta)) & \leq & \int_{B_x(5nr_n(c,\beta)^d) \cap U} 
 (nr_n(c,\beta)^d)^{1-d} \mathsf{E} \Bigg( n \exp(-cn \|B_O(r_n(c,\beta)) \cap \cC(n,r_n(c,\beta))\|)  \\
& &  \exp\left( -c\eta |w|\frac{\|(B_{w(nr_n(c,\beta)^{d-1})^{-1}}(r_n(c,\beta)) \setminus B_O(r_n(c,\beta))) \cap \cC(n,r_n(c,\beta))\|}{\|B_{w(nr_n(c,\beta)^{d-1})^{-1}}(r_n(c,\beta)) \setminus B_O(r_n(c,\beta))\|} \right) \Bigg) \ dw \\
& \leq & (nr_n(c,\beta)^d)^{1-d} \EXP{W_n(r_n(c,\beta))}  \to 0,
\end{eqnarray*}
as $n \to \infty$, since by Lemma \ref{lem:cgs_exp_isoln}, $\EXP{W_n(r_n(c,\beta))} \rar \beta$ and $nr_n(c,\beta)^d \to \infty$ as $n \to \infty.$ 
\remove{We have shown that for $i=1,2$, $I_{in}(r_n(c,\beta)) \rar 0,$ and hence by Lemma \ref{lem:po_approx},
\[ d_{TV}(W_n(r_n(c,\beta)),Po(\EXP{W_n(r_n(c,\beta))})) \rar 0, \] 
as $n \rar \infty$. Again, since $\EXP{W_n(r_n(c,\beta))} \rar \beta$, we have $Po(\EXP{W_n(r_n(c,\beta))}) \stackrel{d}{\rar} Po(\beta)$. Consequently,
$d_{TV}(W_n(r_n(c,\beta)),Po(\beta)) \to 0$  as $n \rar \infty$. 
As convergence in \emph{total variation distance} implies convergence in distribution, we get (\ref{eqn:po_convg}).} \qed

\remove{
By assumption since $a > a(c),$ by continuity, we can choose a $A_0 > 1$ such that $a \eta(a,A_0c) >1.$
Define the event $A_i := \{\mbox{$X_i$ connects to all points of $\Pn \cap B_{X_i}(R_n)$ in $G(n,cn,ar_n)$}\}.$  Let $N_n = \Pn([0,1]^2)$. We need to show that $\pr{\cap_{i=1}^{N_n} A_i} \to 1$ as $n \to \infty$. Let $O$ denote the origin and recall that we are using the toroidal metric.
\bea
\pr{\cup_{i=1}^{N_n} A_i^c} & \leq &  \sum_{i=1}^{n+n^{\frac{3}{4}}} \pr{A_i^c}  + \pr{|N_n-n| > n^{\frac{3}{4}}} \no \\
 & \leq & (n+n^{\frac{3}{4}})p_n + \pr{|N_n-n| > n^{\frac{3}{4}}},
\label{eqn:prob_Ai}
\eea
where $p_n = \pr{\mbox{$O$ does not connect to at least one point of $\Pn \cap B_0(R_n)$ in $G(n,cn,ar_n)$} }$. From \cite[Lemma 1.4]{Penrose03}, we know that the second term on the right in (\ref{eqn:prob_Ai}) is summable. Conditioning on the number of points of $\P_n$ in $B(O,R_n)$ and then using the Boole's inequality, we get
\beas
p_n & \leq &  \sum_{k=0}^{\infty} \frac{(n\theta_d R_n^d)^ke^{-n\theta_d R_n^d}}{k!} \frac{k}{\theta_d R_n^d} \int_{B_O(R_n)} e^{-cn\|B_0(r_n)\cap B_x(r_n)\|} dx \\
& \leq & \sum_{k=0}^{\infty} \frac{(n\theta_d R_n^d)^ke^{-n\theta_d R_n^d}}{k!} \frac{k}{\theta_d R_n^d} \int_{B_O(R_n)} e^{-cn\|B_0(r_n)\cap B_{x_n}(r_n)\|} dx \\
& \leq &  \sum_{k=0}^{\infty} \frac{(n\theta_d R_n^d)^ke^{-n\theta_d R_n^d}}{k!} \frac{k}{\theta_d R_n^d} \int_{B_O(R_n)} e^{-cn\zeta(a,A_0c) \theta_d r_n^d} dx, \\
& = & n\theta_d R_n^d e^{-cn\zeta(a,A_0c) \theta_d r_n^d},
\eeas
where $\theta_d r_n^d \zeta(a,A_0c) = \inf \{ \|B_O(r_n) \cap B_{x_n}(r_n)\| : x_n \in B_O(R_n) \}.$ Since the infimum is attained for a point on the boundary of $B_O(R_n)$,
\begin{equation}
\zeta(a,A_0c) \geq \eta(a,A_0c), \qquad d \geq 2,
\label{eqn:zeta}
\end{equation}
where $\eta(a,c)$ is as defined in (\ref{eqn:defn_eta}). A proof of this is given in the appendix. Simplifying by substituting for $R_n$ and $r_n$, we get
\beas
(n+n^{3/4})p_n & \leq & (n+n^{3/4}) n\theta_d R_n^d e^{-cn\eta(a,A_0c) \theta_d r_n^d} \\
& = & (n+n^{3/4}) \frac{A_0\log n}{n^{a\eta(a,A_0c)}} \to 0,
\eeas
as $n \to \infty$ since $a\eta(a,A_0c) > 1$. \qed }

%\section{Appendix}
\label{sec:appendix}

\remove{
\subsection{Derivation of Equation \ref{eqn:prob_isoln}. } We shall prove that given a bounded random closed set $F$, we have that $\pr{\Xi \cap F = \emptyset} = \exp\{-\mu \EXP{\|F\|}\}$. The derivation now follows easily from the fact that $\EXP{\|B_O(r) \cap C(\lam,r)\|} = \theta_d r^d(1-\exp\{-\lam \theta_d r^d\})$ (see \cite[Chapter 3]{Hall88}). Let $F \subset B_0(t)$ a.s. for some $t > 0$. Now given that exactly $N$ points of $\Xi$ lie in $B_0(t)$, the probability that none of the points lie in $F$ is given by $(1 - \frac{\EXP{\|F\|}}{\theta_d t^d})^N$.  Since $N$ is Poisson random variable with mean $\mu \theta_d t^d$, we obtain that
\[ 
\pr{\Xi \cap F = \emptyset} = \EXP{(1 - \frac{\EXP{\|F\|}}{\theta_d t^d})^N} = \exp\{-\mu \EXP{\|F\|}\}. \] \qed
}
% \subsection{Proof of Lemma~\ref{lem:po_approx}. } 

\remove{
\subsection{Proof of (\ref{eqn:defn_eta}).}

\begin{figure}[tb]
  \centerline{\includegraphics[scale=0.5]{fig2.eps}}
  \caption{$|x| = R, \phi = \phi(\frac{r}{2R})$ and $L(r,R)$ is the area of the lens of intersection, the shaded region.}
  \label{fig:2}
\end{figure}

We sketch the derivation here first for $d=2$ and then for $d \geq 3$. We sketch them seperately as we obtain tighter bounds for $d=2$ than in higher dimensions. Refer to Figure \ref{fig:2} for this derivation when $d=2$. Let $L(r,R) = \|B_O(r) \cap B_x(r)\|, |x| = R$ be the area of lens of intersection of balls of radius $r$ centered $R ( < 2r)$ apart. It is clear that $L(r,R)$ is cut into two equal halves by the line $PQ$ and the area of each of those halves is the area enlosed between the chord in the ball $B_0(r)$ and its boundary. The area of the angular sector $OPQ$ (with $PQ$ considered as the arc along the circumference of the circle) is easily seen to be $\phi(r,R)r^2$. The area of the triangle $OPQ$ is $r \sin(\phi(r,R)) \times r \cos(\Phi^{(1)}(r,R)) = \frac{r^2}{2}\sin(2\Phi^{(1)}(r,R)).$ And, now it is clear that $L(r,R) = (2\Phi^{(1)}(r,R) - \sin(2\Phi^{(1)}(r,R)))r^2$ and that $\cos(\Phi^{(1)}(r,R) = \frac{R}{2r}$. Now for $d \geq 3$. The width of the lens of intersection is $2r-R$. Thus the lens of intersection contains a ball of diameter $2r-R$. The volume of such a ball is $\theta_d(r-\frac{R}{2})^d$ . Thus $L(r,R) > \theta_d(r-\frac{R}{2})^d$. Since $\zeta(a,A_0c) = L(r_n,R_n)$, the corresponding expressions for $\eta(a,A_0c)$ follows.  \qed
}


\begin{thebibliography}{99}

%\footnotesize

\bibitem[Appel and Wierman 1987]{Appel87}
{\sc Appel, M. J. and Wierman, J. C.} (1987). On the absence of infinite AB percolation clusters in bipartite graphs. {\em J. Phys. A: Math. Gen.} {\bf 20}, 2527-2531.

\bibitem[Benjamini and Kesten 1995]{Benjamini95}
{\sc Benjamini, I and Kesten, H.} (1995). Percolation of arbitrary
words in $\{0,1\}^{\mN}$.  {\em Ann. Prob.}  {\bf 23,}
1024-1060.

\bibitem[Dousse et al. 2006]{Dousse06}
{\sc Dousse, O., Franceschetti, M.,Macris, N.,Meester, R. and Thiran, P.} (2006). Percolation in the Signal to Interference Ratio Graph.
{\em J. Appl. Prob.} {\bf 43}, 552-562.

\bibitem[Franceschetti et al. 2007]{Franc07}
{\sc Franceschetti, M., Dousse, O., Tse D. N. C. and Thiran  P.} (2007). Closing the gap in the capacity of wireless networks via percolation theory.  {\em IEEE Trans. Info. Theory}  {\bf 53(3)}, 1009-1018.

\bibitem[Goldstein and Penrose 2010]{Goldstein10}
{\sc Goldstein, L. and Penrose, M.D.} (2010). Normal approximation for coverage models over binomial point processes. {\em Ann. Appl. Prob.} {\bf 20}, 696-721. 

\bibitem[Gilbert 1961]{Gilbert61}
{\sc Gilbert, E.~N.} (1961) Random plane networks. {\em SIAM J.} {\bf 9}, 533--543.

\bibitem[Grimmett 1999]{Grimmett99}
{\sc Grimmett, G.} (1999). {\em Percolation}, Springer-Verlag, Heidelberg.

\bibitem[Gupta and Kumar 2000]{Gupta00}
{\sc Gupta, P. and Kumar, P.R.} (2000). The capacity of wireless networks. {\em IEEE Trans. Info. Theory} {\bf 46(2)}, 388-404.

\bibitem[Hall 1988]{Hall88}
{\sc Hall, P.} (1988). {\em Introduction to the theory of coverage processes}, John Wiley and Sons.

\bibitem[Halley 1980]{Halley80}
{\sc Halley, J.W.} (1980). AB percolation on triangular lattice. {\em Ordering in two dimensions}, ed. S. Sinha (Amsterdam: North-
Holland), 369-371.

\bibitem[Halley 1983]{Halley83}
{\sc Halley, J.W.} (1983). Polychromatic percolation. {\em Percolation structures and processes}, ed G. Deutscher, R. Zallen and J.
Adler (Bristol: Adam Hilger), 323-351.

\bibitem[Kesten et al. 1998]{Kesten98}
{\sc Kesten, H., Sidoravicius, V. and Zhang, Y.} (1998). Almost All Words Are Seen At Critical Site Percolation On The Triangular Lattice. {\em Elec. J. Prob.} {\bf 4}, 1-75.

\bibitem[Kesten et al. 2001]{Kesten01}
{\sc Kesten, H., Sidoravicius, V. and Zhang, Y.}  (2001). Percolation
of arbitrary words on the close-packed graph of $\mZ^2$.
{\em Elec. J. Prob.}  {\bf 6,} 1-27.

\bibitem[Meester and Roy 1996]{Meester96}
{\sc Meester, R. and Roy, R.} (1996). {\em Continuum Percolation}, Cambridge University Press.

\bibitem[Moran 1973]{Moran73}
{\sc Moran, P.A.P.} (1973). The random volume of interpenentrating sphers in space. {\em J. Appl. Prob.} {\bf 10}, 837-846. 

\bibitem[Penrose 2003]{Penrose03}
{\sc Penrose, M.D.} (2003). {\em Random Geometric Graphs}, Oxford University Press, New York.

\bibitem[Scheinerman and Wierman 1987]{Scheinerman87}
{\sc Scheinerman, E. R. and Wierman, J. C.} (1987). Infinite AB clusters exist. {\em J. Phys. A: Math. Gen.} {\bf 20}, 1305 -1307.

\bibitem[Sevsek et al. 1983]{Sevsek83}
{\sc Sevsek, F., Debierre, J.M. and Turban, L.} (1983). Antipercolation on Bethe and triangular lattices. {\em J. Phys. A: Math. Gen.} {\bf 16}, 801-810.

\bibitem[Tanemura 1996]{Tanemura96}
{\sc Tanemura, H.} (1996). Critical behaviour for a continuum percolation model. In {\em Probability Theory and Mathematical Statistics : Proceedings of the Seventh Japan-Russia Symposium,
Tokyo 1995.} (eds S. Watanabe, M. Fukushima, Yu. V. Prohorov and A. N. Shiryayev). World Scientific, River Edge NJ, 485-495. 

\bibitem[Tse and Vishwanath 2005]{Tse05}
{\sc Tse, D. and Vishwanath, P.} (2005).
{\em Fundamentals of Wireless Communication}, Cambridge University Press.

\bibitem[Wierman and Appel 1987]{Wierman87}
{\sc Wierman, J. C. and Appel, M. J.} (1987). Infinite AB percolation clusters exist on the triangular lattice. {\em J. Phys. A: Math. Gen.} {\bf 20}  
, 2533- 2537

\bibitem[Wu and Popov 2003]{Wu04}
{\sc Wu, X-Y and Popov, S. Yu} (2004). On AB bond percolation on the square lattice and AB site percolation on its line graph.
 {\em J.  Statist. Phys.}, {\bf 110}, no. 1-2, 443--449.

\end{thebibliography}
\end{document}